\documentclass[a4paper]{amsart}
\usepackage{amsmath,amssymb,amsthm,enumerate,graphicx,mathrsfs}
\usepackage{hyperref}

\title[Ambient and Poincar\'e metrics for CR $3$-manifolds]
{Ambient and Poincar\'e metrics for CR $3$-manifolds associated with the self-dual Einstein ACH metric}

\author{TAIJI MARUGAME}
\date{}

\renewcommand\a{\alpha}
\renewcommand\b{\beta}
\newcommand\g{\gamma}
\renewcommand\d{\delta}
\newcommand\e{\epsilon}
\renewcommand\th{\theta}

\newcommand\pa{\partial}
\newcommand\ol{\overline}
\newcommand{\wt}{\widetilde}
\newcommand{\wh}{\widehat}

\newtheorem{lem}{Lemma}[section]

\newtheorem{theorem}[lem]{Theorem}
\newtheorem{prop}[lem]{Proposition}
\newtheorem{cor}[lem]{Corollary}
\theoremstyle{definition}
\newtheorem{dfn}[lem]{Definition}
\newtheorem{rem}[lem]{\it Remark}

\numberwithin{equation}{section}

\makeatletter
\@addtoreset{equation}{section}

\makeatother
\address{Department of Mathematics, The University of Electro-Communications, 1-5-1 Chofugaoka, Chofu, Tokyo 182-8585, Japan}
\email{marugame@uec.ac.jp}



\begin{document}

\begin{abstract} 
We construct new ambient and Poincar\'e metrics for the Fefferman conformal manifold over $3$-dimensional CR manifolds, starting from the associated self-dual Einstein ACH metric. Unlike the classical ambient and Poincar\'e metrics arising from approximate solutions to the complex Monge-Amp\`ere equation, these metrics satisfy Einstein-Maxwell-type equations to infinite order rather than the Ricci-flat or Einstein equations. 
Since these metrics are determined to infinite order without ambiguity, they enable us to construct CR invariant differential operators and local CR invariants involving arbitrarily high order derivatives of the Tanaka-Webster curvature and torsion. As an application, we obtain the ambient metric construction of the CR GJMS (Gover-Graham) operators of all orders in dimension $3$.  
\end{abstract}
\maketitle


\section{Introduction}

The Fefferman-Graham ambient  metric and the Poincar\'e metric (\cite{FG}) are canonical geometric structures associated with conformal manifolds.
For the standard conformal sphere $S^m$, realized as the projectivization of the null cone $\mathcal{N}$ in the Minkowski space $\mathbb{R}^{m+1,1}$, the ambient metric is the Minkowski metric on $\mathbb{R}^{m+1,1}$
and the Poincar\'e metric is the hyperbolic metric on the hyperboloid or on the Poincar\'e ball having $S^m$ as the boundary at infinity.   
For a general conformal manifold $(\mathcal{C}^m, [g])$ of signature $(p, q)$, the null cone $\mathcal{N}$ is replaced by the metric bundle $\mathbb{R}_+\to\mathcal{G}\to \mathcal{C}$ and one constructs the ambient metric as a pseudo-Riemannian metric $\wt g$ of signature $(p+1, q+1)$ on $\wt{\mathcal{G}}\cong \mathcal{G}\times(-\e_0, \e_0)_\rho$ satisfying the asymptotic Ricci-flat equation
\[
\mathrm{Ric}(\widetilde g)=\begin{cases}
O(\rho^\infty) & (m: \textrm{odd}) \\
O(\rho^{m/2-1}) & (m:  \textrm{even}).
\end{cases}
\]
Restricting $\wt g$ to an appropriate hypersurface asymptotic to $\mathcal{G}\times\{0\}$, one obtains the associated Poincar\'e metric of signature $(p+1, q)$ which satisfies the asymptotic Einstein equation.  The expansion of $\wt g$ in $\rho$ is determined to infinite order when the dimension $m$ is odd. However, for even $m$, an obstruction to formally solving the Ricci-flat equation gives rise to an ambiguity in higher order terms of $\wt g$ and the corresponding Poincar\'e metric.

The ambient metric provides a powerful tool for constructing differential invariants and invariant differential operators on conformal manifolds. Among these, the GJMS (Graham-Jenne-Mason-Sparling) operators \cite{GJMS} form a particularly important family. These are conformally invariant linear differential operators acting on conformal densities whose principal parts are  powers $\Delta^k$ of the Laplacian. Although such operators are constructed for all $k$ when $m$ is odd, one has the restriction $k\leq m/2$ in even-dimensional cases due to the ambiguity of $\wt g$, and it is proved by Graham \cite{G1} (for $m=4, k=3$) and Gover-Hirachi \cite{GH} (for general cases) that there are no such invariant operators for $k>m/2$. Thus, there can be no refinement of the ambient metric in general even-dimensional conformal geometry that yields higher order GJMS operators.
Nevertheless, one may still hope for a certain refinement of the ambient metric within a special class of even-dimensional conformal manifolds, namely the Fefferman conformal manifolds associated with strictly pseudoconvex CR manifolds.

The Fefferman metric $[g^{\rm F}_\th]$ is a Lorentzian conformal metric which is canonically defined on a principal $S^1$-bundle $\mathcal{C}$ over a strictly pseudoconvex CR manifold $(M^{2n+1}, H, J)$. We can obtain CR invariant objects on $M$ by applying conformally invariant constructions to the Fefferman space $(\mathcal{C}, [g^{\rm F}_\th])$, but since $\dim \mathcal{C}=2n+2$ is even, the constructions are restricted due to the ambiguity of the ambient metric. 
In particular, the ambient metric construction of the CR GJMS operators, i.e., CR invariant linear differential operators acting on CR densities $\mathcal{E}(w, w'), w-w'\in\mathbb{Z}$, whose principal parts are powers of the sublaplacian $\Delta_b^k\ (k=w+w'+n+1)$, works only up to $k=n+1$. 
However, in contrast to the case of general even-dimensional conformal manifolds, Gover-Graham \cite{GG} proved the existence of CR invariant powers of sublaplacian for $k>n+1$ with $(w, w')\notin \mathbb{N}\times \mathbb{N}$ by using the CR tractor calculus.  
In dimension $3$ ($n=1$), they further used the BGG machinery to treat the exceptional cases $(w,w')\in\mathbb N\times\mathbb N$ and  obtained the operators for all admissible weights. More precisely, they proved the following.
\begin{theorem}[{\cite[Theorem 1.3]{GG}}]\label{GJMS-three-dim}
Let $(M, H, J)$ be a $3$-dimensional strictly pseudoconvex CR manifold. For any $(w, w')\in \mathbb{C}^2$ with $w-w'\in\mathbb{Z}$, $k:=w+w'+2\in\mathbb{N}_{>0}$, there exists a CR invariant linear differential operator 
\[
P_{w, w'}\colon \mathcal{E}(w, w')\longrightarrow \mathcal{E}(w-k, w'-k)
\]
whose principal part is $\Delta_b^k$.
\end{theorem}

The aim of this paper is to construct a new ambient metric for the Fefferman space $(\mathcal{C}^4, [g^{\rm F}_\th])$ over CR $3$-manifolds without ambiguity and give an alternative proof to Theorem \ref{GJMS-three-dim}. Since an ambient metric can be reconstructed from the corresponding Poincar\'e metric, we begin by constructing a conformally compact metric $g'$ whose conformal infinity is the Fefferman space. 

To this end, we start with an asymptotically complex hyperbolic (ACH) Einstein metric which has $(M^3, H, J)$ as the boundary at infinity. The ACH metric is the CR analogue of the Poincar\'e metric in conformal geometry and is modeled on the complex hyperbolic metric
whose boundary at infinity is the CR sphere. When $M$ is the boundary of a strictly pseudoconvex domain $\Omega\subset \mathbb{C}^{n+1}$, an ACH Einstien metric is given by Cheng-Yau's complete K\"ahler-Einstein metric on $\Omega$ constructed via the (approximate) solution $\rho$ to the complex Monge-Amp\`ere equation (\cite{CY, F}). The solution $\rho$ also yields a K\"ahler potential of the ambient metric for the Fefferman space over $M$, and this construction is in fact the historical origin of the ambient metric construction in conformal geometry.

By Matsumoto \cite{Mat1, Mat2}, the construction of Einstein ACH metrics has been generalized for partially integrable CR manifolds, but the metric also has an ambiguity arising from an obstruction to solving the Einstein equation formally. However, the obstruction vanishes in the case of integrable CR structures. This suggests the possibility of constructing  an ACH metric without ambiguity by imposing an additional condition which is compatible with the Einstein equation. When $\dim M=3$, the author \cite{Mar} proved that the self-dual equation serves as such a condition, and constructed an ACH metric $g$ on $\ol X\cong M\times [0, \e)_r$ satisfying
\[
\mathrm{Ric}(g)+\frac{3}{2}g=O(r^\infty), \quad W^-(g)=O(r^\infty).
\] 
Unlike the approximate Cheng-Yau metric, this ACH metric is determined uniquely to  infinite order in $r$, and can be used to reprove Theorem \ref{GJMS-three-dim} in the case where $w=w'$; one of our motivations to construct an ambient metric corresponding to $g$ is to remove this restriction.

To construct a Poincar\'e metric for $(\mathcal{C}^4, [g^{\rm F}_\th])$ from the self-dual Einstein ACH metric $g$, we introduce the notion of ``$S^1$-invariant Poincar\'e metrics''. Since the conformal structure $[g^{\rm F}_\th]$ is $S^1$-invariant, it is natural to consider a principal $S^1$-bundle 
\[
\pi\colon \ol{X'}\longrightarrow \ol X,\quad \bigl(\,\ol{X'}\cong\mathcal{C}\times[0, \e)_r,\ \ol X\cong M\times [0, \e)_r\,\bigr)
\]
with the boundary $\mathcal{C}\to M$, and a conformally compact $S^1$-invariant Lorentzian metric $g'$ on $X'\cong \mathcal{C}\times(0, \e)_r$ whose conformal infinity is $(\mathcal{C}, [g^{\rm F}_\th])$. We say $g'$ is an {\it $S^1$-invariant Poincar\'e metric} if it satisfies the asymptotically hyperbolic condition and the multiple of the infinitesimal generator of $S^1$-action
\[
K:=\frac{3}{2}\frac{\pa}{\pa s}\in T\ol{X'}
\]
satisfies $g'(K, K)=-1$; see \S\ref{S1-inv-Poincare} for the precise definition. It is proved that an (even) $S^1$-invariant Poincar\'e metric can be written in the form 
\[ 
g'=-\omega^2+\pi^* \underline{g}
\]
with the 1-form $\omega:=-K\lrcorner\, g'$ and an ACH metric $\underline{g}$ on the base space $X$. We call $\underline{g}$ the {\it induced ACH metric}
and impose the condition that this agrees with the self-dual Einstein ACH metric $g$. 
To determine $\omega$, we consider the 2-form 
\[
F:=d\omega=\frac{1}{2}F_{ab}\th^a\wedge\th^b, \quad F_{ab}=-2\nabla'_a K_b,
\]
which descends to a $2$-form on $X$, and  impose the self-dual equation 
\begin{equation}\label{F-intro}
F^-=O(r^\infty)
\end{equation}
with respect to $\underline{g}(=g)$.
We thereby obtain an $S^1$-invariant Poincar\'e metric $g'$ whose expansion in $r$ is uniquely determined (Theorem \ref{normalization-g-prime}).  As a consequence of the Einstein equation for $\underline{g}$ and the equation \eqref{F-intro}, the metric $g'$ satisfies the Einstein-Maxwell-type equation
\begin{equation}\label{E-M-intro}
E'_{ab}:=R'_{ab}+\frac{3}{2}g'_{ab}-\frac{1}{2}F_{ac}F_b{}^c-\frac{1}{4}
(|F|^2-6)K_a K_b=O(r^\infty).
\end{equation}
We note that our situation is similar to the Kaluza-Klein theory in general relativity.
The self-dual equations $W^-=O(r^\infty), F^-=O(r^\infty)$ can also be translated into intrinsic tensorial equations $L^-_{abcd}=O(r^\infty),  F^-_{ab}=O(r^\infty)$ on $X'$; see \S\ref{translation-self-dual} for the definitions of $L^-_{abcd}$, $F^-_{ab}$. The metric $g'$ is characterized by these equations, and we obtain the following theorem:  

\begin{theorem}\label{Poincare}
Let $(M, H, J)$ be a $3$-dimensional strictly pseudoconvex CR manifold and $(\mathcal{C}, [g^{\mathrm{F}}_\th])$ the Fefferman space over $M$. Then, there exists an $S^1$-invariant even Poincar\'e metric $g'$ on $X'=\mathcal{C}\times(0, \e)_r$ with the conformal infinity $(\mathcal{C}, [g^{\mathrm{F}}_\th])$ which satisfies 
\[
E'_{ab}=O(r^\infty), \quad L^-_{abcd}=O(r^\infty), \quad F^-_{ab}=O(r^\infty).
\]
The metric $g'$ is unique modulo $O(r^\infty)$ up to pull-back by $S^1$-equivariant diffeomorphisms of $\mathcal{C}\times[0, \e)_r$ fixing points on $\mathcal{C}\times\{0\}$. Moreover, the expansion of each component in the normal form of $g'$ with respect to a representative $g^{\mathrm{F}}_\th$ can be expressed by a universal formula in terms of the curvature and torsion of the Tanaka-Webster connection. 
\end{theorem}

As a notion of ambient metric corresponding to $S^1$-invariant Poincar\'e metrics, we consider  
$S^1$-invariant straight pre-ambient metrics $\wt g$ on the ambient space $\wt{\mathcal{G}}\cong\mathcal{G}\times(-\e_0, \e_0)_\rho$, where $\mathcal{G}\cong \mathcal{C}\times (0, \infty)_t$ is the metric bundle over the Fefferman space. We let $\wt K$ be $3/2$ times the infinitesimal generator of $S^1$-action on $\wt{\mathcal{G}}$, and assume that $\wt g$ satisfies the equation
\[
 \wt g(\wt K, \wt K)=\frac{\wt \rho}{4},
\]
where $\wt\rho:=\wt g(E, E)$ is the squared norm of the Euler field $E=t\partial_t$ on $\wt{\mathcal{G}}$. Then, restricting $\wt g$ to the hypersurface $\mathcal{H}=\{\wt\rho=-1\}\subset \wt{\mathcal{G}}$ and changing the variable as $\rho=-\frac{1}{2}r^2$, we have an $S^1$ invariant Poincar\'e metric $g'$ on $\mathcal{H}\cong \mathcal{C}\times (0, \e)_r$. Requiring $g'$ to coincide with the metric constructed in Theorem \ref{Poincare}, we obtain a uniquely determined ambient metric for the Fefferman space $(\mathcal{C}^4, [g^{\rm F}_\th])$.
The equations characterizing $g'$ can be reformulated in terms of curvature quantities of $\wt g$. In particular, the equation \eqref{E-M-intro} is equivalent to the equation
\[
\wt E_{AB}:=\wt R_{AB}+\frac{2}{\wt\rho}(\wt F_{AC}\wt F_B{}^C-\wt g_{AB})
-\frac{4}{\wt\rho^{\,2}}(|\wt F|^2-6)\wt K_A \wt K_B=O(\rho^\infty),
\]
where $\wt F_{AB}:=-2\wt\nabla_A \wt K_B$. See \eqref{F-prime}, \eqref{L-tilde}, \eqref{tilde-L-minus}, \eqref{tilde-F-minus} for the definitions of the ambient tensors $\wt L^-_{ABCD},  (\wt F')^-_{AB}$ corresponding to $L^-_{abcd},  F^-_{ab}$.  As a result, we obtain the following theorem:

\begin{theorem}\label{ambient}
Let $(M, H, J)$ be a $3$-dimensional strictly pseudoconvex CR manifold and $(\mathcal{C}, [g^{\mathrm{F}}_\th])$ the Fefferman space over $M$. Then, there exists an $S^1$-invariant straight pre-ambient metric $\wt g$ on $\wt{\mathcal{G}}=\mathcal{G}\times(-\e_0, \e_0)_\rho$ which satisfies 
\[
 \wt g(\wt K, \wt K)=\frac{\wt \rho}{4}
\]
and
\[
\wt E_{AB}=O(\rho^\infty), \quad \wt L^-_{ABCD}=O(\rho^\infty), \quad (\wt F')^-_{AB}=O(\rho^\infty).
\]
The metric $\wt g$ is unique modulo $O(\rho^\infty)$ up to pull-back by $\mathbb{C}^*$-equivariant diffeomorphisms of $\wt{\mathcal{G}}$ fixing points on $\mathcal{G}\times\{0\}$. Moreover, the expansion of each component in the normal form of $\wt g$ with respect to a representative $g^{\mathrm{F}}_\th$ can be expressed by a universal formula in terms of the curvature and torsion of the Tanaka-Webster connection. 
\end{theorem}

Applying the GJMS construction to our ambient metric $\wt g$, we obtain an alternative proof of Theorem \ref{GJMS-three-dim}; see \S\ref{GJMS-construction}.

Finally, we briefly discuss another potential application of the new ambient metric.
In conformal geometry, 
Weyl invariants of the ambient metric, namely linear combinations of complete contractions of tensors formed from covariant derivatives of the curvature of $\wt g$, give rise to local scalar conformal invariants. In odd dimensions, every local conformal invariant arises in this way. In even dimensions, however, the conformal weight is subject to a restriction due to the ambiguity of the ambient metric (\cite{FG}). The situation in CR geometry is analogous to that in even-dimensional conformal geometry, and we cannot produce all local scalar CR invariants by using the ordinary ambient metric.  Since our new ambient metric is determined to infinite order, we can construct local CR invariants which involve high order derivatives of the curvature and torsion via Weyl invariants composed of $\wt K, \wt \nabla\wt K, \wt \nabla^{(m)}\wt R$ with no restriction on $m$. It is natural to ask whether all local CR invariants in dimension $3$ arise in this manner. We hope to revisit this problem in future work.

\medskip

This paper is organized as follows: In \S2, we review basic notions of CR geometry and the Fefferman conformal structure. In \S3, we explain the definition of ACH metrics and recall the existence theorem for self-dual Einstein ACH metrics. In \S4, we introduce $S^1$-invariant Poincar\'e metrics and the induced ACH metric, and construct  the Poincar\'e metric $g'$ appearing in Theorem \ref{Poincare}. We then express the curvature of $g'$ in terms of that of the induced ACH metric and derive the equations characterizing $g'$, thereby completing the proof of Theorem \ref{Poincare}. In \S5, we describe the ambient metric $\wt g$ corresponding to $g'$, and compute the equations characterizing it. Finally, we apply the GJMS construction to $\wt g$ to obtain an alternative proof of Theorem \ref{GJMS-three-dim}. 

\medskip

{\it Notation}: 
Throughout this paper, we adopt Einstein's summation convention. We use different types of indices for tensors on various spaces according to the following convention:
\begin{center}
\begin{tabular}{ll}
$\alpha,\beta,\gamma,\ldots$ & for $T^{1,0}M$,\\
$i,j,k,\ldots$ & for $TX$ or ${}^\Theta T\overline{X}$,\\
$a,b,c,\ldots$ & for $TX'$,\\
$A,B,C,\ldots$ & for $T\widetilde{\mathcal{G}}$.
\end{tabular}
\end{center}
Skew symmetrization over a collection of indices is denoted by square brackets. For example,
\[
B_{[ijk]}=\frac{1}{3!}(B_{ijk}+B_{jki}+B_{kij}-B_{jik}-B_{ikj}-B_{kji}).
\]
We identify the exterior product $\omega\wedge\eta$ and the symmetric product $\omega\cdot\eta$ of $1$-forms with the corresponding bilinear forms via
\begin{equation*}
\begin{aligned}
(\omega\wedge\eta)(V, W)&=\omega(V)\eta(W)-\omega(W)\eta(V), \\
(\omega\cdot\eta)(V, W)&=\frac{1}{2}(\omega(V)\eta(W)+\omega(W)\eta(V)).
\end{aligned}
\end{equation*}

\section{CR manifolds and the Fefferman metric}
\subsection{CR structures and the Tanaka-Webster connection}
A {\it CR structure} on a $(2n+1)$-dimensional $C^\infty$ manifold $M$ is a pair $(H, J)$ of a contact distribution $H\subset TM$ and an almost complex structure $J\in\mathrm{End}(H)$ on $H$. Let $\mathbb{C}H=T^{1, 0}M\oplus T^{0,1}M$ denote the $\pm i$ eigenspace decomposition for $J$. Then, $(H, J)$ is said to be {\it partially integrable} if 
\[
[\Gamma(T^{1, 0}M), \Gamma(T^{1, 0}M)]\subset \Gamma(\mathbb{C}H),
\]
and {\it integrable} if 
\[
[\Gamma(T^{1, 0}M), \Gamma(T^{1, 0}M)]\subset \Gamma(T^{1, 0}M).
\]
Note that a CR structure is always integrable in the $3$-dimensional case ($n=1$). Throughout this paper, we assume that $(H, J)$ is integrable.

A contact form $\th\in\Gamma(H^\perp)\subset\Gamma(TM)$ determines the  {\it Reeb vector field} $T\in \Gamma(TM)$,  characterized by the conditions $\th(T)=1, T\lrcorner\, d\th=0$. 
Given a local frame $(Z_\a)$ for $T^{1, 0}M$, the collection 
\[
(T, Z_\a, Z_{\ol\a}:=\ol{Z_\a})
\]
forms a local frame for $\mathbb{C}TM$, which we call an {\it admissible  frame}. Its dual frame $(\th, \th^\a, \th^{\ol\a})$ is called an {\it admissible coframe}. The {\it Levi form} associated with $\th$ is a $J$-invariant symmetric form on $H$ defined by 
\[
h_\th(X, Y)=d\th(X, JY)\quad X, Y\in H.
\]
Extending $h_\th$ complex bilinearly to $\mathbb{C}H$, we write $h_{\a\ol\b}:=h_\th(Z_\a, Z_{\ol\b})$. Then, we have 
\[
d\th=ih_{\a\ol\b}\th^\a\wedge\th^{\ol\b}.
\]

We say the CR structure $(H, J)$ is {\it strictly pseudoconvex} if there exists a contact form $\th$ for which the Levi form $h_\th$ is positive definite. Such a contact form is unique up to multiplication by a positive function $e^{\Upsilon}, \Upsilon\in C^\infty(M)$, and  the Levi form transforms conformally according to $h_{e^\Upsilon\th}=e^\Upsilon h_\th$. We assume that our CR structure is strictly pseudoconvex, and use the Levi form $h_{\a\ol\b}$ and its inverse $h^{\a\ol\b}$ to raise and lower indices according to 
\[
W^\a=h^{\a\ol\b}W_{\ol\b},\quad W^{\ol\a}=h^{\b\ol\a}W_{\b},\quad W_\a=h_{\a\ol\b}W^{\ol\b},\quad W_{\ol\a}=h_{\b\ol\a}W^\b.
\]

A choice of contact form $\th$ determines a linear connection $\nabla^{\mathrm{TW}}$ on $TM$ called the {\it Tanaka-Webster connection}, which is characterized by the following properties:
\begin{itemize}
\item The Reeb vector field is parallel: $\nabla^{\mathrm{TW}} T=0$.
\item If we extend $\nabla^{\mathrm{TW}}$ to $\mathbb{C}TM$, it preserves $T^{1, 0}M$ and the Levi form: 
\[
\nabla^{\mathrm{TW}}_X\bigl(\Gamma(T^{1, 0}M)\bigl)\subset\Gamma(T^{1, 0}M),\ \nabla^{\mathrm{TW}} h_\th=0.
\]
\item The connection 1-form $\omega_\b{}^\a$ defined by $\nabla^{\mathrm{TW}} Z_\b=\omega_\b{}^\a\otimes Z_\a$ satisfies the structure equation 
\begin{equation}\label{structure-equation}
d\th^\a=\th^\b\wedge\omega_\b{}^\a+A^\a{}_{\overline \b}\,\th\wedge\th^{\overline \b}
\end{equation}
with some tensor $A^\a{}_{\ol\b}$.
\end{itemize}
The tensor $A^\a{}_{\ol\b}$ appearing in \eqref{structure-equation} is called the {\it Tanaka-Webster torsion}, and it satisfies $A_{\a\b}=A_{\b\a}$, where $A_{\a\b}:=\ol{A_{\ol\a\ol\b}}$. 
The curvature forms $\Omega_\a{}^\b=d\omega_\a{}^\b-\omega_\a{}^\g\wedge\omega_\g{}^\b$ can be written as 
\begin{equation*}
\begin{aligned}
\Omega_\a{}^\b&=R_\a{}^\b{}_{\g\overline\mu} \th^\g\wedge\th^{\overline \mu}+\nabla^{\mathrm{TW}\, \b} A_{\a\g}\th^\g\wedge\th-\nabla^{\mathrm{TW}}_\a A^\b{}_{\overline \g}\th^{\overline \g}\wedge\th \\
& \quad -iA_{\a\g}\th^\g\wedge\th^\b+ih_{\a\overline\g}A^\b{}_{\overline\mu}
\th^{\overline\g}\wedge\th^{\overline\mu},
\end{aligned}
\end{equation*}
where $R_\a{}^\b{}_{\g\overline\mu}$ is the {\it Tanaka-Webster curvature tensor}. The {\it  Tanaka-Webster Ricci tensor} and the {\it Tanaka-Webster scalar curvature} are defined by 
\[
R_{\a\ol\b}:=R_\g{}^\g{}_{\a\ol\b}, \quad R:=R_\g{}^\g
\]
respectively. We also define the {\it Schouten tensor} by 
\[
P_{\a\ol\b}:=\frac{1}{n+2}\Bigl(R_{\a\ol\b}-\frac{R}{2(n+1)}h_{\a\ol\b}\Bigr)
\]
and denote its trace by $P=P_{\a}{}^\a=\frac{1}{2(n+1)}R$. Note that when $\dim M=3$ we have $R_{1\ol1 1\ol1}=Rh_{1\ol1}h_{1\ol1}$, $R_{1\ol1}=Rh_{1\ol1}$, and $P_{1\ol1}=\frac{1}{4}Rh_{1\ol1}$.

The Tanaka-Webster connection depends on the choice of contact form $\th$, and its transformation formula under the rescaling $\th\mapsto \wh\th=e^\Upsilon \th$ is described by a formula involving derivatives of $\Upsilon$; see, e.g., \cite{GG}.

\subsection{CR densities and natural differential operators}
The {\it CR canonical bundle} is the complex line bundle $K_M:=\wedge^{n+1}(T^{0,1}M)^\perp\subset\wedge^{n+1}\mathbb{C}T^\ast M$ over $M$. We locally choose  an 
$(n+2)$-nd root $\mathcal{E}(1,0)$ of $K_M^{-1}$, and define the {\it CR density bundles} by
\begin{equation}\label{CR-density}
\mathcal{E}(w, w^\prime)=\mathcal{E}(1,0)^{\otimes w}\otimes\overline{\mathcal{E}(1,0)}{}^{\otimes w^\prime}
\end{equation}
for $(w, w^\prime)\in\mathbb{C}^2, w-w^\prime\in\mathbb{Z}$. 
When $w=w'$, the bundle $\mathcal{E}(w,w')$ is independent of the choice of $\mathcal{E}(1,0)$ and is therefore globally defined. We call a (local) section of $\mathcal{E}(w, w')$ a {\it CR density of weight $(w, w')$}, and denote the space of such sections by the same symbol $\mathcal{E}(w, w')$. A CR density $f\in\mathcal{E}(w, w')$ can be identified with a complex-valued function $f$ on $\mathcal{E}(-1, 0)\setminus\{0\}$ satisfying the homogeneity condition
\[
\d_\lambda^* f=\lambda^w\bar{\lambda}^{w'}f\qquad \lambda\in\mathbb{C}^*,
\] 
where $\d_\lambda$ denotes the natural $\mathbb{C}^*$-action on $\mathcal{E}(-1, 0)$. The Tanaka-Webster connection $\nabla^{\mathrm{TW}}$  induces a linear connection on the CR density bundles, which we also denote by $\nabla^{\mathrm{TW}}$. 

For each choice of contact form $\th$, there exists a section $\zeta_\th\in K_{M}=\mathcal{E}(-n-2, 0)$ satisfying
\begin{equation}\label{volume-normalized}
\th\wedge(d\th)^n=i^{n^2}n!\th\wedge(T\lrcorner\,\zeta_\th)\wedge(T\lrcorner\,\overline\zeta_\th);
\end{equation}
such a section is uniquely determined up to multiplication by a $U(1)$-valued function. Then the CR density
\[
\tau_\th:=|\zeta_\th|^{-\frac{2}{n+2}}\in\mathcal{E}(1, 1)
\]
is uniquely determined by $\th$ and satisfies $\tau_{\wh\th}=e^{-\Upsilon}\tau_\th$ for $\wh\th=e^{\Upsilon}\th$. Thus, the {\it weighted Levi form}  and its inverse 
\[
\boldsymbol{h}_{\a\ol\b}=\tau_\th h_{\a\ol\b}, \quad 
\boldsymbol{h}^{\a\ol\b}=\tau_\th^{-1} h^{\a\ol\b}
\]
are invariant under changes of contact form.
 It can be shown that  $\tau_\th$ is parallel with respect to $\nabla^{\mathrm{TW}}$ and so is $\boldsymbol{h}_{\a\ol\b}$; see \cite[Proposition 2.1]{GG}. Conversely, a choice of positive section $\tau\in\mathcal{E}(1, 1)$ defines a unique contact form; such a section is called a {\it CR scale}.
 
We say a linear differential operator $P_\th \colon \mathcal{E}(w, w')\to \mathcal{E}(w-k,  w'-k)$ associated to each contact form $\th$ is a {\it natural differential operator} if it can be expressed as  a linear combination of complete contractions of 
\[
\boldsymbol{h}^{\a\ol\b}, \quad \nabla^{\mathrm{TW}}_\a,\quad   \nabla^{\mathrm{TW}}_{\ol\a},\quad \nabla^{\mathrm{TW}}_0\bigl(:=\tau_\th^{-1}\nabla^{\mathrm{TW}}_T\bigr), \quad  R_{\a}{}^\b{}_{\g\ol\mu},\quad A_{\a\b}, \quad A_{\ol\a\ol\b}.
\]
Here the covariant derivatives may also be applied to the curvature tensors and the torsion tensors, and $k$ is the number of $\boldsymbol{h}^{\a\ol\b}, \nabla^{\mathrm{TW}}_0$ appearing in each summand. For example, the {\it sublaplacian}
\[
\Delta_b:=-\boldsymbol{h}^{\a\ol\b}\bigl(\nabla^{\mathrm{TW}}_\a\nabla^{\mathrm{TW}}_{\ol\b}+\nabla^{\mathrm{TW}}_{\ol\b}\nabla^{\mathrm{TW}}_{\a}\bigr)\colon \mathcal{E}(w, w')\to \mathcal{E}(w-1, w'-1)
\]
is a natural differential operator.  We say a natural differential operator $P_\th$ is a {\it CR invariant differential operator} if it is independent of the choice of contact form: $P_{\wh\th}=P_\th$.

\subsection{The Fefferman metric} 
Let $(M, H, J)$ be a strictly pseudoconvex CR manifold of dimension $2n+1$. 
The Fefferman metric associated with $(M, H, J)$ was originally introduced by Fefferman \cite{F} in the setting of boundaries of strictly pseudoconvex domains in $\mathbb{C}^{n+1}$. An intrinsic construction for abstract CR manifolds was later developed by \cite{BDS, Lee, Fa}. In this paper, we follow the formulation of Lee \cite{Lee}.

We define the  {\it Fefferman space} as the principal $S^1$-bundle 
\[
\mathcal{C}:=(K_M\setminus\{0\})/\mathbb{R}_+
\]
over $M$.
We fix a contact form $\th$ and take an admissible coframe $(\th, \th^\a, \th^{\ol\a})$. 
Let $s\in\mathbb{R}/2\pi\mathbb{Z}$ denote the local fiber coordinate of $\mathcal{C}$ defined by the local section 
\[
\zeta=(\det h_{\a\ol\b})^{1/2}\th\wedge\th^1\wedge\cdots\wedge\th^n.
\]
Then, it can be shown that the real $1$-form 
\begin{align*}
\sigma&:=\frac{2}{n+2}\bigl(ds-{\rm Im}\,\omega_\a{}^\a-
P\th\bigr) \\
&=\frac{2}{n+2}\Bigl(ds+i\omega_\a{}^\a-\frac{i}{2}h^{\a\ol\b}d h_{\a\ol\b}-P\th\Bigr)
\end{align*}
is independent of the choice of $(\th^\a)$ and becomes a globally defined $1$-form on $\mathcal{C}$ determined by $\th$. Note that our $\sigma$ is twice the 1-form denoted by $\sigma$ in \cite{Lee}. Its exterior derivative descends to a $2$-form on $M$ and is given by
\begin{equation}\label{d-sigma}
d\sigma=2i P_{\a\ol\b}\th^\a\wedge\th^{\ol\b}+C_\a\th^\a\wedge\th+C_{\ol\a}\th^{\ol\a}\wedge\th,
\end{equation}
where 
\[
C_\a:=\frac{2}{n+2}\bigl(i\nabla^{\mathrm{TW}}_{\ol\g}A^{\ol\g}{}_\a-\nabla^{\mathrm{TW}}_\a P\bigr)
\]
and $C_{\ol\a}:=\ol{C_\a}$.
With respect to the local coframe $(\th, \th^\a, \th^{\ol\a}, \sigma)$, we define the {\it Fefferman
metric} on $\mathcal{C}$ by  
\[
g^{\rm F}_\th:=2\bigl(h_{\a\ol\b}\th^\a\cdot\th^{\ol\b}+\theta\cdot \sigma\bigr).
\]
This is an $S^1$-invariant Lorentzian metric and is equal to {\it twice} the Fefferman metric in the convention of \cite{Lee}.
For another choice of contact form $\wh\th=e^\Upsilon \th$, we have $g^{\rm F}_{\wh\th}=e^\Upsilon g^{\rm F}_\th$. Therefore, the conformal structure $[g^{\rm F}_\th]$ depends only on the CR structure $(H, J)$.


\section{$\Theta$-manifolds and ACH metrics}

We briefly review $\Theta$-structures and ACH metrics on manifolds with boundary, following the framework introduced by Epstein-Melrose-Mendoza \cite{EMM}.  Our exposition is mainly based on \cite{Mat1, Mat2}.
\subsection{$\Theta$-structures and $\Theta$-vector fields}
Let $X$ be the interior of a $(2n+2)$-dimensional manifold $\ol X$ with boundary $M$. We assume that $M$ is endowed with a strictly pseudoconvex CR structure $(H, J)$. A {\it $\Theta$-structure} on  $\ol X$ is the conformal class $[\Theta]$ of a 1-form $\Theta\in\Gamma(M, T^* \ol X)$ along $M$ whose restriction to $TM$ gives a contact form with positive Levi form on $M$. We call $(\ol X, [\Theta])$ a {\it $\Theta$-manifold}. Let $(\ol X', [\Theta'])$ be another $\Theta$-manifold with the same boundary $M$, and let  $f\colon (\ol X, [\Theta])\to (\ol X', [\Theta'])$ be a $C^\infty$-map which restricts to the identity map on $M$. We say $f$ is a {\it $\Theta$-diffeomorphism} if it restricts to a diffeomorphism between neighborhoods of $M$ and satisfies $[f^*\Theta']=[\Theta]$ on $M$.  

Let $r\colon \ol X\to [0, \infty)$ be a boundary defining function on $\ol X$; that is, $r^{-1}(0)=M$, $dr|_M\neq0$.
A vector field $V$ on $(\ol X, [\Theta])$ is called a {\it $\Theta$-vector field} if it satisfies
\[
V=O(r), \quad \wt\Theta(V)=O(r^2),
\]
where $\wt\Theta\in\Gamma(\ol X, T^*\ol X)$ is an arbitrary extension of some $\Theta\in[\Theta]$. Note that the condition is independent of the choices of a representative $\Theta$ and its extension. 
Choose a local coframe $(dr,\wt\Theta,\alpha^1,\dots,\alpha^{2n})$ near $M$ and let
$(N,T,Y_1,\dots,Y_{2n})$ denote the dual frame. Then, every $\Theta$-vector field can be written locally as a $C^\infty(\ol X)$-linear combination of
\begin{equation}\label{Theta-frame}
rN,\ r^2T, \ rY_1,\dots,rY_{2n}.
\end{equation}
On $X$, these vector fields form a local frame for $TX$, and 
one can check that the transition functions between such local frames  extend smoothly up to the boundary. Thus, we obtain a vector bundle ${}^\Theta T\ol X$ over $\ol X$, called the {\it $\Theta$-tangent bundle}, whose sections are identified with $\Theta$-vector fields.  Note that the restriction ${}^\Theta T\ol X|_X$ is canonically isomorphic to $TX$. 

A fiber metric of the $\Theta$-tangent bundle is called a {\it $\Theta$-metric}. We also have the dual ${}^\Theta T^*\ol X$ and its exterior power $\wedge^k({}^\Theta T^*\ol X)$, whose sections are called  {\it $\Theta$-$k$-forms}. 
\subsection{Admissible defining functions associated with a contact form}
Since the set of $\Theta$-vector fields vanishing at a fixed point $p\in M$ forms an ideal of the Lie algebra of the $\Theta$-vector fields, the fiber ${}^\Theta T_p \ol X$ inherits a Lie algebra structure from the Lie bracket. In terms of the $\Theta$-frame \eqref{Theta-frame}, the derived series of ${}^\Theta T_p \ol X$ is 
\[
{}^\Theta T_p \ol X\supset D_p^1\supset D_p^2\supset \{0\},
\]
where 
\[
D_p^1=\langle (r^2T)_p, \ (rY_1)_p,\dots,(rY_{2n})_p\rangle, 
\quad D_p^2=\langle (r^2T)_p \rangle.
\]
For a given $\Theta$-metric $g$, we define subbundles
$R\subset {}^\Theta T \ol X|_M$ and $L\subset D^1$ by the orthogonal decompositions: 
\[
{}^\Theta T_p \ol X=R_p\overset{\perp}{\oplus} D_p^1, \quad D_p^1=L_p\overset{\perp}{\oplus}  D_p^2.
\] 
The section 
\[
\Bigl(\frac{dr}{r}\Bigr)\Big|_M\in \Gamma(M, {}^\Theta T^*\ol X)
\]
is independent of the choice of boundary defining function $r$, and its restriction to $R$ provides a coframe for $R$. We denote by $\nu\in \Gamma(R)$  the dual frame:
\[
\Bigl(\frac{dr}{r}\Bigr)(\nu)=1.
\] 

Using the $\Theta$-metric $g$, we can associate a $1$-jet of a boundary defining function along $M$ with each choice of contact form $\theta$ as follows:  Choose $\Theta\in[\Theta]$ such that $\Theta|_{TM}=\theta$, and extend it arbitrarily to a $1$-form $\wt\Theta$ on $\ol X$. For a boundary defining function $r$, the section 
\[
(r^{-2}\wt\Theta)|_{D^2}\in\Gamma(M, (D^2)^*)
\]
is independent of the choice of extension $\wt\Theta$ and depends only on the $1$-jet of $r$ along $M$. We say that $r$ is an {\it admissible boundary defining function} associated with $\theta$ 
if the section $t\in\Gamma(D^2)$ determined by
\[
(r^{-2}\wt\Theta)(t)=1
\]
satisfies $|t|_g=1$.
This condition normalizes the $1$-jet of $r$. Moreover, if $\wh\th=e^{\Upsilon}\th$ is another contact form, then any admissible boundary defining function $\wh r$ associated with $\wh\th$ satisfies 
\[
\wh r=e^{\wt\Upsilon/2}r+O(r^2),
\]
where $\wt\Upsilon\in C^\infty(\ol X)$ is an arbitrary extension of $\Upsilon$.
\subsection{ACH metrics}
Let $r$ be an admissible boundary defining function associated with a contact form $\th$. Then, we obtain a well-defined linear map
\[
\lambda_\th\colon H_p\longrightarrow L_p, \quad 
Y\longmapsto \pi((r\wt Y)_p),
\]
where $\wt Y\in\Gamma(T\ol X)$ is an arbitrary extension of $Y$ and $\pi\colon {}^\Theta T_p \ol X\to L_p$ denotes the orthogonal projection with respect to $g_p$. One checks that $\lambda_\th$ is a linear isomorphism and that it is determined by the $1$-jet of $r$ and hence depends only on $\th$.
Via this identification, we define ACH metrics as follows:
\begin{dfn}\label{ACH-def}
A $\Theta$-metric $g$ on $(\ol X, [\Theta])$ is called an {\it asymptotically complex hyperbolic (ACH) metric} with CR infinity $(M, H, J)$ if it satisfies the following conditions:
\begin{enumerate}
\item[(i)]  For any boundary defining function $r$, we have 
\[
\Bigl|\dfrac{dr}{r}\Bigr|^2_g=\dfrac{1}{4}\quad \textrm{on}\ M,
\]
or, equivalently, $|\nu|^2_g=4$. 
\item[(ii)] For any $p\in M$, the map
\[
L_p\longrightarrow {}^\Theta T_p \ol X, \quad Z\longmapsto [\nu_p, Z]
\]
gives the identity map on $L_p$.
\item[(iii)] For any contact form $\th$ on $M$, we have 
\[
\lambda_\th^*(g|_{L})=2h_{\a\ol\b}\th^\a\cdot\th^{\ol\b},
\]
where the right-hand side denotes the Levi form on $H$ associated with $\th$.
\end{enumerate}
We call $(\ol X, [\Theta], g)$ an {\it ACH manifold}.
\end{dfn}
\subsection{Normal form ACH metrics}
For any CR manifold $(M, H, J)$, one can associate to the product space $M\times[0, \e)_r$ a canonical $\Theta$-structure $[\Theta]_{\mathrm{std}}$, called the {\it standard $\Theta$-structure} characterized by the condition $\Theta(\pa/\pa r)=0$. Given an  admissible frame $(T, Z_\a, Z_{\ol\a})$ for $\mathbb{C}TM$, we define a (complex) $\Theta$-frame $(\mbox{\boldmath $Z$}_i)$ on $M\times[0, \e)_r$ by
\begin{equation}\label{std-frame}
\mbox{\boldmath $Z$}_\infty=r\pa_r, \quad \mbox{\boldmath $Z$}_0=r^2 T, 
\quad \mbox{\boldmath $Z$}_\a=r Z_\a, \quad \mbox{\boldmath $Z$}_{\overline \a}=r Z_{\overline \a},
\end{equation}
where $\pa_r=\pa/\pa r$, and $T, Z_\a, Z_{\ol\a}$ are extended trivially to $M\times[0, \e)_r$. The dual $\Theta$-coframe $(\boldsymbol{\th}^i)$ is given by 
\begin{equation}\label{std-coframe}
\boldsymbol{\th}^\infty=\frac{dr}{r}, \quad \boldsymbol{\th}^0=\frac{\th}{r^2}, \quad \boldsymbol{\th}^\a=\frac{\th^\a}{r}, \quad 
\boldsymbol{\th}^{\ol\a}=\frac{\th^{\ol\a}}{r}.
\end{equation}

Let $g$ be a $\Theta$-metric on a $\Theta$-manifold $(\ol X, [\Theta])$ and let $\th$ be a contact form on the boundary $M$. By solving a first-order PDE, one can construct a $\Theta$-diffeomorphism
\[
f\colon (M\times[0, \e)_r, [\Theta]_{\mathrm{std}})\longrightarrow (\ol X, [\Theta])
\]
such that the components $g_{ij}:=(f^*g)(\mbox{\boldmath $Z$}_i, \mbox{\boldmath $Z$}_j)$ satisfy 
\[
g_{\infty\infty}=4, \quad g_{\infty 0}=g_{\infty\a}=0, \quad g_{00}=1+O(r);
\]
see \cite[Lemma 4.9]{Mat1} or \cite[Lemma 2.5]{Mat2}. The germ of $f$ along $M$ is uniquely determined by these conditions. Via $f$, we regard $r$ as a function on $\ol X$ and call it the {\it model defining function}. Note that the equation $g_{00}=1+O(r)$ corresponds to the condition that $r$ be admissible. In this setup, the ACH condition for $g$ can be characterized as follows:

\begin{prop}[{\cite[Proposition 4.12]{Mat1}, \cite[Proposition 2.7]{Mat2}}]\label{normal-ACH-condition}
A $\Theta$-metric $g$ is ACH if and only if 
\[
g_{0\a}=O(r), \qquad g_{\a\b}=O(r), \qquad g_{\a\overline{\b}}=h_{\a\overline{\b}}+O(r).
\]
\end{prop}
For an ACH metric $g$, its pullback $f^*g$ on $M\times[0, \e)_r$ is    called a {\it normal form ACH metric}.  
We say that an ACH metric $g$ is {\it even} if, for every contact form $\th$, its normal form can be written as
\[
f^* g=\frac{4d r^2+h_r}{r^2}
\]
with a family of Riemannian metrics $h_r$ on $M$ admitting an even Laurent expansion in $r$. When we write $f^*g$ in terms of the standard $\Theta$-coframe \eqref{std-coframe}, the evenness is equivalent to the condition that $g_{00}, g_{\a\b}, g_{\a\ol\b}$ are even in $r$, and $g_{0\a}$ is odd in $r$.
\subsection{The self-dual Einstein ACH metric}
For any strictly pseudoconvex partially integrable CR manifold $(M, H, J)$, Matsumoto \cite{Mat1, Mat2}  constructed an ACH metric $g$ on $(M\times [0, \e)_r, [\Theta]_{\mathrm{std}})$ that solves the Einstein equation $\mathrm{Ric}+\frac{n+2}{2}g=0$ to finite order in $r$. Using this metic, he further generalized the CR GJMS operators to the partially integrable setting (\cite{Mat3}). 

Although the obstruction arising in the construction of a formal solution to the Einstein equation vanishes for integrable CR structures,
an additional normalization condition is still needed to remove the ambiguity in the formal solution. In the case where $M$ is 3-dimensional, the author \cite{Mar} showed that the self-dual equation serves as such a normalizaiton condition, thereby refining Matsumoto's  result as follows:
\begin{theorem}[{\cite{Mar}}]\label{self-dual-ACH}
Let $(M, H, J)$ be a $3$-dimensional strictly pseudoconvex CR manifold. Then, there exists an even  ACH metric $g$ on $(M\times [0, \e)_r, [\Theta]_{\mathrm{std}})$ with CR infinity $(M, H, J)$ which satisfies 
\[
\mathrm{Ric}(g)+\frac{3}{2}g=O(r^\infty), \quad W^-(g)=O(r^\infty).
\] 
The metric $g$ is unique modulo $O(r^\infty)$ up to pull-back by $\Theta$-diffeomorphisms. Moreover, when $g$ is written in  normal form with respect to a contact form $\th$ on $M$, the expansion of each coefficient $g_{ij}$ in the $\Theta$-frame \eqref{std-frame} can be expressed by a universal formula in terms of the curvature and torsion of the Tanaka-Webster connection associated with $\th$. 
\end{theorem}

We refer the reader to \cite{B} for a twistor construction of $g$ when $(M, H, J)$ is real analytic.

\section{Poincar\'e metric for the Fefferman metric}

\subsection{$S^1$-invariant Poincar\' e metrics}\label{S1-inv-Poincare}
Let $(M, H, J)$ be a $(2n+1)$-dimensional strictly pseudoconvex CR manifold and let $(\mathcal{C}, [g^{\mathrm{F}}_\th])$ denote the Fefferman space over $M$. Let $\ol X\cong M\times[0, \e)$ be a $(2n+2)$-dimensional manifold with the boundary $\pa \ol X=M$ and $X\cong M\times(0, \e)$ its interior. We consider a principal $S^1$-bundle $\pi\colon \ol{X'}\to \ol X$ such that $\pa \ol {X'}=\mathcal{C}$ as a principal $S^1$-bundle over $M$ and denote its interior by $X'\to X$.  
\begin{dfn} A Lorentzian metric $g'$ on $X'$ is called an {\it $S^1$-invariant Poincar\'e metric} with conformal infinity $(\mathcal{C}, [g^{\mathrm{F}}_\th])$ if the following conditions are satisfied:
\begin{itemize}
\item[(i)] $g'$ is $S^1$-invariant.
\item[(ii)] The Killing vector field 
\[
K:=\frac{n+2}{2}\frac{\pa}{\pa s}
\]
satisfies 
\[
g'(K, K)=-1.
\]
Here, $\pa/\pa s$ is the infinitesimal generator of the $S^1$-action on $X'$. In particular, the integral curves of $K$ are geodesics.
\item [(iii)]$g'$ is a conformally compact metric with conformal infinity $(\mathcal{C}, [g^{\rm F}_\th])$, that is, for any boundary defining function $r\colon \ol{X'}\to [0, \infty)$, the metric $r^2 g'$ extends smoothly to a Lorentizan metric on $\ol{X'}$ satisfying $[(r^2 g')|_{T\mathcal{C}}]=[g^{\rm F}_\th]$.
\item[(iv)] $g'$ is ``asymptotically hyperbolic'' in the sense  that it satisfies
\[
\Bigl|\frac{dr}{r}\Bigr|^2_{g'}\bigl(=|dr|^2_{r^2 g'}\bigr)=\frac{1}{4}\quad {\rm on}\ \mathcal{C}
\] 
for an arbitrary boundary defining function $r\colon \ol{X'}\to [0, \infty)$.
\end{itemize}
\end{dfn}
In condition (iv) above, we adopt the nonstandard normalization constant $1/4$ instead of the more usual value $1$ in order to ensure compatibility with the normalization of ACH metrics on $\ol X$ in Definition \ref{ACH-def}(i). We also note that $K$ satisfies $\sigma(K)=1$ along the boundary $\mathcal{C}$.

Let $g'$ be an $S^1$-invariant Poincar\'e metric on $X'$. For each representative metric $g^{\rm F}_\th\in[g^{\rm F}_\th]$, there exists a boundary defining function $r$ on a neighborhood of $\mathcal{C}$ in $\ol{X'}$ satisfying
\[
(r^2 g')|_{T\mathcal{C}}=g^{\rm F}_\th, \quad \Bigl|\frac{dr}{r}\Bigr|^2_{g'}=\frac{1}{4}\ 
\text{near}\ \mathcal{C}.
\]
Such an $r$ is unique in a neighborhood of $\mathcal{C}$ and is called the {\it special defining function}; see {\cite[Lemma 2.1]{G}}. Shrinking $\ol{X'}$ if necessary and using the flow of $\mathrm{grad}_{r^2g'} r$, we obtain a diffeomorphism $\ol{X'}\cong\mathcal{C}\times 
[0, \epsilon)_r$. Under this identification, $g'$ takes the normal form
\[
g'=\frac{4dr^2+g_r}{r^2}
\]
where $g_r$ is a smooth one-parameter family of Lorentzian metrics on $\mathcal{C}$. We say $g'$ is {\it even} 
if $g_r$ admits even Taylor expansion in $r$. The evenness plays an important role when we construct the ambient metric from a Poincar\'e metric; see \S\ref{ambient-Poincare}.
\subsection{The induced ACH metric}
We will show that an $S^1$-invariant Poincar\'e metric $g'$ on $X'$ naturally induces a $\Theta$-structure and a $\Theta$-metric on the base manifold $\ol{X}$. 

Let $r\colon \ol{X'}\to[0, \infty)$ be an $S^1$-invariant boundary defining function and $\ol{g'}=r^2g'$ the compactified metric, which is an $S^1$-invariant Lorentzian metric on $\ol{X'}$. Then, since $K$ is $S^1$-invariant and null along the boundary $\mathcal{C}$ with respect to $\ol{g'}$, the $1$-form 
$\bigl(K\lrcorner\, \ol{g'}\bigr)|_\mathcal{C}\in \Gamma(T^*\ol{X'}|_\mathcal{C})$ descends to a section of $T^*\ol X|_M$ whose conformal class is independent of the choice of $r$. Moreover, $K$ satisfies $K\lrcorner\, g^{\rm F}_\th=\th$ on $\mathcal{C}$. Thus, we can define the {\it induced $\Theta$-structure} on $\ol X$ by 
\[
[\Theta]:=\bigl[\bigl(K\lrcorner\, \ol{g'}\bigr)\big|_\mathcal{C}\bigr].
\]

Next, we define a metric on $X$. Since $K$ is timelike on $X'$, the tangent space at each point $x\in X'$ admits the orthogonal decomposition
\[
T_xX'=\mathbb{R}K_x\oplus K_x^\perp
\]
with respect to $g'$. We then define a Riemannian metric $\underline g$ on $X$ by identifying $T_{\pi(x)}X$ with $K_x^\perp$ via the linear isomorphism
\[
\pi_*\colon K_x^\perp\to T_{\pi(x)}X
\]
and transporting the metric $g'|_{K_x^\perp}$ to $T_{\pi(x)}X$.
Since both $g'$ and $K$ are $S^1$-invariant, the resulting metric is independent of the choice of the point $x$ in the fiber.

If we define an $S^1$-invariant $1$-form $\omega$ on $X'$ by 
\[
\omega:=-K\lrcorner\, g',
\]
it satisfies
\[
\omega(K)=1, \quad K^{\perp}={\rm Ker}\,\omega,
\]
and the Poincar\'e metric $g'$ can be written as
\[ 
g'=-\omega^2+\pi^* \underline{g}.
\]
\begin{prop}\label{induced}
The Riemannian metric $\underline{g}$ becomes a $\Theta$-metric on $\ol X$, and the $2$-form $d\omega$ descends to a $\Theta$-$2$-form on $\ol X$.  Moreover, if the Poincar\'e metric $g'$ is even, the $\Theta$-metric $\underline{g}$ is an even ACH metric. 
\end{prop}
\begin{proof}
We fix a contact form $\th$ on $M$ and take the special defining function $r\colon \ol{X'}\to [0, \infty)$ determined by $g^{\rm F}_\th$. Since $g'$ and $g^{\rm F}_\th$ are $S^1$-invariant, so is $r$ by uniqueness, and it descends to a defining function of $M\subset \ol X$ satisfying $|dr/r|^2_g=1/4$ near $M$. Moreover, the diffeomorphism $\ol{X'}\cong\mathcal{C}\times 
[0, \epsilon)_r$ is $S^1$-equivariant if we let $S^1$ act on the
 $\mathcal{C}$-component in the right-hand side. Thus, it also induces a diffeomorphism $\ol X\cong M\times [0, \epsilon)_r$.  We extend $\th, \th^\a, \sigma, h_{\a\ol\b}$ to $\mathcal{C}\times [0, \epsilon)_r$ in the trivial way so that $\sigma(K)=1$ holds on $\mathcal{C}\times [0, \epsilon)_r$. 
Then, by the equation $g'(K, K)=-1$ and the $S^1$-invariance of $g'$, we can write $g'$ on $\mathcal{C}\times [0, \epsilon)_r$ in the normal form
\begin{equation}\label{g-prime-normal-form}
g'=4\frac{dr^2}{r^2}+\frac{1}{r^2}\bigl(2h_{\a\ol\b}\th^\a\cdot\th^{\ol\b}+2\theta\cdot \sigma+r G+2r\sigma\cdot \psi-r^2\sigma^2\bigr)
\end{equation}
with some real tensors 
\begin{align*}
G&=2G_{\a\ol\b}\th^\a\cdot\th^{\ol\b}+G_{\a\b}\th^\a\cdot\th^\b+
G_{\ol\a\ol\b}\th^{\ol\a}\cdot\th^{\ol\b} \\
&\quad +2G_{0\a}\th\cdot\th^\a+2G_{0\ol\a}\th\cdot\th^{\ol\a}+G_{00}\th^2,
\\
\psi&=\psi_\a\th^\a+\psi_{\ol\a}\th^{\ol\a}+\psi_0\th
\end{align*}
defined on $M\times [0, \epsilon)_r$. Since $(K\lrcorner\, r^2g')|_{r=0}=\th$,  the induced $\Theta$-structure on $M\times [0, \epsilon)_r$ agrees with the standard $\Theta$-structure. The $1$-form $\omega=-K\lrcorner\, g'$ and the metric 
$\underline{g}=g'+\omega^2$ are given by 
\begin{align}
\omega&=\sigma-\frac{\th}{r^2}-\frac{\psi}{r}, \label{omega} \\
\underline{g}&=4\Bigl(\frac{dr}{r}\Bigr)^2+2h_{\a\ol\b}\frac{\th^\a}{r}\cdot\frac{\th^{\ol\b}}{r}
+\Bigl(\frac{\th}{r^2}\Bigr)^2+\frac{1}{r}G+\Bigl(\frac{\psi}{r}\Bigr)^2+2\frac{\th}{r^2}\cdot\frac{\psi}{r}. \label{underline-g}
\end{align}
Thus, $\underline{g}$ is a $\Theta$-metric with $r$ serving as the model defining function for $\th$, and $d\omega$ is a 
$\Theta$-$2$-form on $M\times[0, \epsilon)_r$. 
The boundary value of  $\underline{g}$ is given by
\begin{align*}
\underline{g}|_M&=4\Bigl(\frac{dr}{r}\Bigr)^2+2h_{\a\ol\b}\frac{\th^\a}{r}\cdot\frac{\th^{\ol\b}}{r}
+\Bigl(\frac{\th}{r^2}\Bigr)^2+\Bigl(\psi_\a|_M \frac{\th^\a}{r}+\psi_{\ol\a}|_M\frac{\th^{\ol\a}}{r}\Bigr)^2 \\
&\quad +2\frac{\th}{r^2}\cdot\Bigl(\psi_\a|_M \frac{\th^\a}{r}+\psi_{\ol\a}|_M\frac{\th^{\ol\a}}{r}\Bigr).
\end{align*}
When $g'$ is even, we have  $\psi_\a=O(r)$ so $\underline{g}$ is an ACH metric with the infinity $(M, H, J)$ by Proposition \ref{normal-ACH-condition}. Moreover, the evenness of $g'$ implies that of $\underline{g}$ by the formula \eqref{underline-g}.
\end{proof}

We call $\underline{g}$ the {\it induced $\Theta$-metric} associated with $g'$. If it is an ACH metric, we call it the {\it induced ACH metric}.

By \eqref{underline-g}, the components of $\underline{g}$ in the $\Theta$-coframe \eqref{std-coframe} are given by 
\begin{equation}\label{g-bar-component}
\begin{aligned}
\underline{g}_{\a\ol\b}&=h_{\a\ol\b}+r G_{\a\ol\b}+\psi_\a\psi_{\ol\b}, \\
\underline{g}_{\a\b}&=r G_{\a\b}+\psi_\a\psi_\b, \ \\
\underline{g}_{0\a}&=r^2G_{0\a}+r\psi_0\psi_\a+\psi_\a,  \\
 \underline{g}_{00}&=1+r^3G_{00}+r^2\psi_0^2+2r\psi_0. 
\end{aligned}
\end{equation}

We define a skew-symmetric 2-tensor $F_{ab}$ on $X'$ by writing $d\omega$ as
\[
d\omega=\frac{1}{2}F_{ab}\th^a\wedge\th^b, \quad F_{ab}=-F_{ba}.
\]
Since $\omega=-K_a \th^a$ and $K$ is a Killing vector field, we have
\[
F_{ab}=-2\nabla'_a K_b,
\]
where $\nabla'$ is the Levi-Civita connection of $g'$. 
By Proposition \ref{induced}, $d\omega$ also determines a $\Theta$-$2$-tensor $F_{ij}=F_{[ij]}$ on $\overline X$ via
\[
d\omega=\frac{1}{2}F_{ij}\boldsymbol{\th}^i\wedge\boldsymbol{\th}^j, 
\]
where $(\boldsymbol{\th}^i)$ is an arbitrary $\Theta$-coframe.
Note that the equation $K^a F_{ab}=0$ can also be seen from the geodesic equation $\nabla'_KK=0$. By \eqref{d-sigma}, we have
\begin{align*}
d\omega&=d\sigma-d\Bigl(\frac{\th}{r^2}\Bigr)-d\Bigl(\frac{\psi}{r}\Bigr) \\
&=2i P_{\a\ol\b}\th^\a\wedge\th^{\ol\b}+C_\a\th^\a\wedge\th+C_{\ol\a}\th^{\ol\a}\wedge\th \\
&\quad -\frac{i}{r^2}h_{\a\ol\b}\th^\a\wedge\th^{\ol\b}+\frac{2}{r^3}dr\wedge\th -\frac{1}{r}d(\psi_\a\th^\a+\psi_{\ol\a}\th^{\ol\a}+\psi_0\th)+\frac{1}{r^2}dr\wedge \psi.
\end{align*}
For later computations, we set
\[
\psi'_0:=r\psi_0.
\]
Then, the components of the $\Theta$-2-tensor $F_{ij}$ with respect to the  $\Theta$-coframe \eqref{std-coframe}
are given by
\begin{equation}\label{Fij}
\begin{aligned}
F_{\a\ol\b}&=2ir^2P_{\a\ol\b}-i(1+\psi'_0)h_{\a\ol\b}-r\bigl(\nabla^{\rm TW}_\a\psi_{\ol\b}-\nabla^{\rm TW}_{\ol\b}\psi_{\a}\bigr), \\
F_{\a\b}&=-r\bigl(\nabla^{\rm TW}_\a\psi_{\b}-\nabla^{\rm TW}_{\b}\psi_{\a}\bigr), \\
F_{\infty0}&=2-(r\pa_r-2)\psi'_0, \\
F_{\infty\a}&=-(r\pa_r-1)\psi_\a, \\
F_{0\a}&=-r^3C_\a-r^2\nabla^{\rm TW}_T \psi_\a-r^2\psi_{\ol\g}A^{\ol\g}{}_\a+r\nabla^{\rm TW}_\a\psi_0'.
\end{aligned}
\end{equation}
The remaining components are obtained from these by skew-symmetry and complex conjugation.

\subsection{Poincar\' e metric associated with the self-dual Einstein ACH metric}
We now specialize to the case in which the CR manifold $M$ is $3$-dimensional. In this case, we can uniquely determine the asymptotic expansion of $g'$ as follows:
\begin{theorem}\label{normalization-g-prime}
Let $(M, H, J)$ be a $3$-dimensional strictly pseudoconvex CR manifold and $(\mathcal{C}, [g^{\mathrm{F}}_\th])$ the Fefferman space over $M$. Then, there exists an $S^1$-invariant even Poincar\'e metric $g'$ on $X':=\mathcal{C}\times(0, \e)_r$ with the conformal infinity $(\mathcal{C}, [g^{\mathrm{F}}_\th])$ which satisfies the following conditions:
\begin{itemize}
\item[(i)] The induced ACH metric $\underline{g}$ on $\ol{X}:=M\times[0, \e)_r$ satisfies the self-dual Einstein equation:
\[
\mathrm{Ric}(\underline{g})+\frac{3}{2}\underline{g}=O(r^\infty), \quad W^-(\underline{g})=O(r^\infty).
\] 
\item[(ii)] The $\Theta$-$2$-tensor $F_{ij}$ on $\ol{X}$ is self-dual with respect to $\underline{g}$: 
\[
F^{-}=O(r^\infty).
\]
\end{itemize}
The metric $g'$ is unique modulo $O(r^\infty)$ up to pull-back by $S^1$-equivariant diffeomorphisms fixing points on $\mathcal{C}$. Moreover, the expansion of each component in the normal form of $g'$ with respect to a representative $g^{\mathrm{F}}_\th$ can be expressed by a universal formula in terms of the curvature and torsion of the Tanaka-Webster connection. 
\end{theorem}
To prove this theorem, we proceed as follows: First let $g$ be  
the self-dual Einstein even ACH metric on $\ol{X}$ provided by Theorem \ref{self-dual-ACH}. We show that the expansions of $\psi'_0(=r\psi_0), \psi_1$ are uniquely determined by the self-duality of $F_{ij}$ with respect to $g$ (rather than $\underline{g}$) together with the initial condition \eqref{psi-initial} below. We then determine the expansion of the tensor $G$ by imposing the equation $\underline{g}=g$.
\medskip

We work in a local $\Theta$-frame \eqref{std-frame} on $\ol{X}$ with $(Z_1)$ normalized by $h_{1\ol1}=1$. Then, for $n=1$, the equations \eqref{Fij}  reduce to
\begin{equation}\label{Fij-three-dim}
\begin{aligned}
F_{1\ol1}&=-\frac{i}{2}(2-r^2 R+2\psi'_0)-r\bigl(\nabla^{\rm TW}_1\psi_{\ol1}-\nabla^{\rm TW}_{\ol1}\psi_1\bigr), \\
F_{\infty0}&=2-(r\pa_r-2)\psi'_0, \\
F_{\infty1}&=-(r\pa_r-1)\psi_1, \\
F_{01}&=-r^3C_1-r^2\nabla^{\rm TW}_T \psi_1-r^2\psi_{\ol1}A^{\ol1}{}_1+r\nabla^{\rm TW}_1\psi_0'.
\end{aligned}
\end{equation}

We recall that the lower-order expansions of the components $g_{ij}, g^{ij}$ are given by 
\begin{equation}\label{lower-order}
\begin{aligned}
g_{00}&=1+O(r^3),  &  g_{01}&=O(r^3), \\ 
g_{1\ol1}&=1-\frac{r^2}{2}R+O(r^3), &  g_{11}&=-2ir^2A_{11}+O(r^3)
\end{aligned}
\end{equation}
and 
\begin{equation*}
\begin{aligned}
g^{00}&=1+O(r^3), & g^{01}&=O(r^3), \\
g^{1\ol1}&=1+\frac{r^2}{2}R+O(r^3), & g^{11}&=-2ir^2A_{\ol1\ol1}+O(r^3);
\end{aligned}
\end{equation*}
see \cite[(7, 3)]{Mat1} or \cite[Proposition 4.1]{Mat2}.  
If $g'$ satisfies the condition (i) in Theorem \ref{normalization-g-prime}, we have $\underline{g}=g+O(r^\infty)$ by uniqueness. Thus, by \eqref{g-bar-component} and \eqref{lower-order}, $\psi'_0$ and $\psi_1$ must satisfy 
\begin{equation}\label{psi-initial}
\psi'_0=O(r^2), \quad \psi_1=O(r^2).
\end{equation}

We choose the orientation of $\ol X$ so that 
\[  
i\,\mbox{\boldmath $\th$}^0\wedge
\mbox{\boldmath $\th$}^1\wedge\mbox{\boldmath $\th$}^{\ol1}\wedge\mbox{\boldmath $\th$}^\infty>0
\]
for the $\Theta$-coframe \eqref{std-coframe}, and write the volume form of $g$ as
\[
vol_g=\frac{1}{4!}\varepsilon_{ijkl}\mbox{\boldmath $\th$}^i\wedge
\mbox{\boldmath $\th$}^j\wedge\mbox{\boldmath $\th$}^k\wedge\mbox{\boldmath $\th$}^l
\]
with a skew-symmetric $\Theta$-tensor $\varepsilon_{ijkl}=\varepsilon_{[ijkl]}$. 
\begin{prop} With respect to the $\Theta$-frame \eqref{std-frame} with $h_{1\ol1}=1$, we have
\begin{equation}\label{volume-form}
\begin{aligned}
\varepsilon_{01\ol1\infty}&=2i-ir^2R+O(r^3), &
\varepsilon_{\infty 1}{}^{01}&=2i+O(r^3), \\
\varepsilon_{\infty 1}{}^{0\ol1}&=-4r^2A_{11}+O(r^3), & \varepsilon_{\infty1}{}^{1\ol1}&=O(r^3), \\
\varepsilon_{\infty 0}{}^{1\ol1}&=2i+ir^2R+O(r^3), & 
\varepsilon_{\infty0}{}^{01}&=O(r^3).
\end{aligned}
\end{equation}
\end{prop}

\begin{proof}
With respect to the $\Theta$-frame $(\mbox{\boldmath $Z$}_i)$, we have
\begin{align*}
\det(g_{IJ})
&=\det \begin{pmatrix}
4 & 0 & 0 & 0 \\
0 & 1 & 0 & 0 \\
0 & 0 & -2ir^2A_{11} & 1-\frac{r^2}{2}R \\
0 & 0 &  1-\frac{r^2}{2}R & 2ir^2A_{\ol1\ol1} 
\end{pmatrix}+O(r^3)  \\
&=-4\Bigl(1-\frac{r^2}{2}R\Bigr)^2+O(r^3)
\end{align*}
and hence
\begin{align*}
vol_g&=|\det(g_{IJ})|^{1/2} i\,\mbox{\boldmath $\th$}^0\wedge
\mbox{\boldmath $\th$}^1\wedge\mbox{\boldmath $\th$}^{\ol1}\wedge\mbox{\boldmath $\th$}^\infty \\
&=i\bigl(2-r^2 R+O(r^3)\bigr)\mbox{\boldmath $\th$}^0\wedge
\mbox{\boldmath $\th$}^1\wedge\mbox{\boldmath $\th$}^{\ol1}\wedge\mbox{\boldmath $\th$}^\infty.
\end{align*}
Thus, we obtain the first equation in \eqref{volume-form}. By raising the indices, we have
\begin{align*}
\varepsilon_{\infty 1}{}^{01}&=g^{00}g^{1\ol1}\varepsilon_{\infty10\ol1}+O(r^3) =\Bigl(1+\frac{r^2}{2}R\Bigr)(2i-ir^2 R)+O(r^3) \\
&=2i+O(r^3), \\
\varepsilon_{\infty 1}{}^{0\ol1}&=
g^{00}g^{\ol1\ol1}\varepsilon_{\infty10\ol1}+O(r^3) 
=2ir^2 A_{11}(2i-ir^2 R)+O(r^3) \\
&=-4r^2A_{11}+O(r^3), \\
\varepsilon_{\infty1}{}^{1\ol1}&=O(r^3), \\
\varepsilon_{\infty 0}{}^{1\ol1}&=g^{1\ol1}g^{\ol1 1}\varepsilon_{\infty0\ol11}+g^{11}g^{\ol1\ol1}\varepsilon_{\infty01\ol1} =\Bigl(1+\frac{r^2}{2}R\Bigr)^2(2i-ir^2 R)+O(r^3) \\
&=2i+ir^2 R+O(r^3), \\
\varepsilon_{\infty0}{}^{01}&=O(r^3).
\end{align*}
\end{proof}
The anti self-dual part of $F_{ij}$ with respect to $g$ is defined by 
\[
F^-_{ij}:=\frac{1}{2}(F_{ij}-{*F}_{ij})\quad  {\rm with}\quad  {*F}_{ij}:=\frac{1}{2}\varepsilon_{ij}{}^{kl}F_{kl}.
\]
Since ${*F}^-_{ij}=-F^-_{ij}$, we have
\begin{align*}
F^-_{01}&=-\frac{1}{2}\varepsilon_{01}{}^{kl}F^-_{kl}\equiv-\varepsilon_{01}{}^{1\infty}F^-_{1\infty}\equiv -\frac{i}{2}F^-_{1\infty}, \\
F^-_{1\ol1}&=-\frac{1}{2}\varepsilon_{1\ol1}{}^{kl}F^-_{kl}
\equiv -\varepsilon_{1\ol1}{}^{0\infty}F^-_{0\infty}\equiv
-\frac{i}{2}F^-_{0\infty} \qquad {\rm mod}\ O(r)\cdot F^-_{ij}.
\end{align*}
We therefore obtain the following lemma.
\begin{lem}\label{self-dual-lemma}
If $F^-_{\infty0}, F^-_{\infty1}=O(r^\infty)$, then $F^-_{ij}=O(r^\infty)$.
\end{lem}

We call terms of the form
\begin{equation*}
O(r)\cdot(r\pa_r)^l \mathcal{D}\psi'_0, \quad O(r)\cdot(r\pa_r)^l \mathcal{D}\psi_1,\quad O(r)\cdot (r\pa_r)^l \mathcal{D}\psi_{\ol1}
\end{equation*}
{\it negligible terms}, where $l\ge 0$ and $\mathcal{D}$ is an $r$-dependent linear differential operator on $M$. Such terms remain unchanged modulo $O(r^{m+1})$ when $\psi'_0, \psi_1$ are perturbed by   terms of order $O(r^m)$. For example, by \eqref{Fij-three-dim} we have 
\[
F_{1\ol1}\equiv-\frac{i}{2}(2-r^2 R+2\psi'_0), \quad F_{01}\equiv-r^3C_1
\]
modulo negligible terms.
\begin{lem}\label{F-psi}
Modulo negligible terms, we have
\begin{equation}\label{F-psi-eq}
\begin{aligned}
2F^-_{\infty0}&\equiv -r\pa_r \psi'_0+O(r^3), \\
2F^-_{\infty1}&\equiv -(r\pa_r-1)\psi_1+O(r^3),
\end{aligned}
\end{equation}
where the terms $O(r^3)$ in the right-hand sides do not involve $\psi_0', \psi_1$.
\end{lem}
\begin{proof}
By \eqref{Fij-three-dim} and \eqref{volume-form}, we have
\begin{align*}
2F^-_{\infty0}&=F_{\infty0}-\frac{1}{2}\varepsilon_{\infty0}{}^{kl}F_{kl} \\
&=F_{\infty0}-\varepsilon_{\infty0}{}^{01}F_{01}-\varepsilon_{\infty0}{}^{0\ol1}F_{0\ol1}-\varepsilon_{\infty0}{}^{1\ol1}F_{1\ol1} \\
&\equiv 2-(r\pa_r-2)\psi'_0-(2i+ir^2R)\Bigl(-i+\frac{i}{2}r^2R-i\psi'_0\Bigr)+O(r^3) \\
&\equiv -r\pa_r \psi'_0+O(r^3)
\end{align*}
and 
\begin{align*}
2F^-_{\infty1}&=F_{\infty1}-\frac{1}{2}\varepsilon_{\infty1}{}^{kl}F_{kl} \\
&=F_{\infty1}-\varepsilon_{\infty1}{}^{01}F_{01}-\varepsilon_{\infty1}{}^{0\ol1}F_{0\ol1}-\varepsilon_{\infty1}{}^{1\ol1}F_{1\ol1} \\
&\equiv -(r\pa_r-1)\psi_1+O(r^3)
\end{align*}
modulo negligible terms.
\end{proof}

By Lemma \ref{self-dual-lemma} and Lemma \ref{F-psi}, we obtain the following proposition:

\begin{prop}
There exist $\psi'_0, \psi_1$ which satisfy \eqref{psi-initial} and 
\[
F^-_{ij}=O(r^{\infty})
\]
with respect to $g$. Such $\psi'_0, \psi_1$ are unique modulo $O(r^\infty)$. Moreover, the expansions of $\psi'_0, \psi_1$ are expressed by universal formulas in terms of Tanaka-Webster curvature quantities. 
\end{prop}
\begin{proof} 
By Lemma \ref{self-dual-lemma}, it suffices to consider the equations $F^-_{\infty0}, F^-_{\infty1}=O(r^\infty)$. For $m\ge 2$, we prove by induction that there exist unique
\[
\psi'^{(m)}_0=\sum_{k=2}^m \phi^{(k)}_0r^k, \quad \psi^{(m)}_1=\sum_{k=2}^m\phi^{(k)}_1r^k
\] 
which satisfy $F^-_{\infty0}, F^-_{\infty1}=O(r^{m+1})$ and that the coefficients $\phi^{(k)}_0, \phi^{(k)}_1$ are expressed in terms of Tanaka-Webster curvature quantities. For $m=2$, by \eqref{psi-initial} and \eqref{F-psi-eq}, we must have $\phi^{(2)}_0= \phi^{(2)}_1=0$ to obtain $F^-_{\infty0}, F^-_{\infty1}=O(r^3)$ since negligible terms are $O(r^3)$. Suppose that we obtain $\psi'^{(m)}_0, \psi^{(m)}_1$ such that $F^-_{\infty0}, F^-_{\infty1}=O(r^{m+1})$. We note that $(r^{-m-1}F^-_{\infty0})|_M$, $(r^{-m-1}F^-_{\infty1})|_M$ are written in terms of Tanaka-Webster curvature quantities by \eqref{Fij-three-dim}. If we add $\phi^{(m+1)}_0r^{m+1}$, 
$\phi^{(m+1)}_1r^{m+1}$ to $\psi'^{(m)}_0, \psi^{(m)}_1$, the changes of $F^-_{\infty0}, F^-_{\infty1}$ are given by 
\begin{align*}
2\delta F^-_{\infty0}&= -(m+1) \phi^{(m+1)}_0r^{m+1}+O(r^{m+2}), \\
2\delta F^-_{\infty1}&= -m\phi^{(m+1)}_1r^{m+1}+O(r^{m+2})
\end{align*}
by \eqref{F-psi-eq}. Hence, we have unique $\phi^{(m+1)}_0, \phi^{(m+1)}_1$ so that  $F^-_{\infty0}, F^-_{\infty1}=O(r^{m+2})$ and these coefficients are written in terms of Tanaka-Webster curvature quantities. Thus, we obtain $\psi'^{(m+1)}_0, \psi^{(m+1)}_1$. Applying Borel's lemma, 
we obtain $\psi'_0$ and $\psi_1$ satisfying $F^-_{ij}=O(r^\infty)$.
\end{proof}
By the proof above, we have $\psi'_0=O(r^3)$.  Hence, by \eqref{g-bar-component} we can uniquely determine the expansion of $G$ so that $\underline{g}=g+O(r^\infty)$. Thus, we obtain an $S^1$-invariant Poincar\'e metric $g'$ satisfying the conditions (i), (ii). We note that the expansion of $g'$ in the normal form \eqref{g-prime-normal-form} is uniquely determined by these conditions. 

We will prove that $g'$ is even. Let $g'_-$ be the Lorentzian metric on 
$\mathcal{C}\times (-\e, 0]_r$ defined by the right-hand side of \eqref{g-prime-normal-form} with the same expansion in $r$ as $g'$.
The reflection 
\[
f\colon \mathcal{C}\times [0, \e) \longrightarrow \mathcal{C}\times (-\e, 0], \quad (x, r)\longmapsto (x, -r)
\]
induces a reflection $M\times (0, \e) \to M\times (-\e, 0)$ which preserves the orientation defined on both sides by 
\[
i\,\mbox{\boldmath $\th$}^0\wedge
\mbox{\boldmath $\th$}^1\wedge\mbox{\boldmath $\th$}^{\ol1}\wedge\mbox{\boldmath $\th$}^\infty 
=ir^{-5}\th\wedge\th^1\wedge\th^{\ol1}\wedge dr>0.
\]
Thus, the Poincar\'e metric $f^* g'_-$ also satisfies the conditions (i), (ii). 
Since the expansions of $\psi_0, \psi_1, G$ are uniquely determined by these conditions, we have $f^* g'_-=g'+O(r^\infty)$, which implies the evenness of $g'$. 

Finally, let $\wh{g}{\,'}$ be another  $S^1$-invariant even Poincar\'e metric on $\mathcal{C}\times(0, \hat\e)_{\hat r}$ which satisfies the conditions (i), (ii) in Theorem \ref{normalization-g-prime}.
 For a representative metric $g^{\mathrm{F}}_\th$, there exists a unique diffeomorphism 
 $\varphi\colon \mathcal{C}\times[0, \e)_{r}\to\mathcal{C}\times[0, \hat\e)_{\hat r}$ fixing points 
 on $\mathcal{C}\times\{0\}$ such that $\varphi^*\wh{g}{\,'}$ is in the normal form \eqref{g-prime-normal-form}. Note that $\varphi$ is $S^1$-equivariant by uniqueness and the $S^1$-invariance of $g^{\mathrm{F}}_\th$, $\wh{g}{\,'}$. Since $\varphi^*\wh{g}{\,'}$ also satisfies (i), (ii), we conclude that $\varphi^*\wh{g}{\,'}=g'+O(r^\infty)$ by the uniqueness of the expansion. This completes the proof of Theorem \ref{normalization-g-prime}.
\begin{rem} Since $\frac{1}{2}F_{ij}\boldsymbol{\th}^i\wedge\boldsymbol{\th}^j$ is closed, the self-dual equation $F^-_{ij}=O(r^\infty)$ implies that $F_{ij}$ is divergence-free with respect to $\underline{g}=g$: 
\begin{equation*}
\nabla^i F_{ij}=O(r^\infty).
\end{equation*} 
Although this condition imposes differential equations on $\psi_0, \psi_1$,  it determines their expansions only up to finite order. 
Hence, the self-dual equation cannot be replaced by this weaker condition.
\end{rem}
\subsection{Relation between the connection and curvature of $g'$ and $\underline{g}$} 
To prove Theorem \ref{Poincare}, we rewrite the conditions (i), (ii) in Theorem \ref{normalization-g-prime} in terms of curvature quantities of $g'$. To this end, we first relate the Levi-Civita connection and the curvature of an $S^1$-invariant Poincar\'e metric $g'$ to those of the induced $\Theta$-metric $\underline{g}$ (in general dimensions). 
 
We set 
\[
N:=\dim X'=2n+3
\]
and let $(\th^i)=(\th^1,\dots, \th^{N-1})$ be an arbitrary coframe on $X$ and $\omega_i{}^j$ the connection 1-forms of the Levi-Civita connection $\nabla$ of $\underline{g}$ with respect to this coframe. Let $\omega'_a{}^b$ be the connection 1-forms of $\nabla'$ with respect to the coframe 
\begin{equation}\label{coframe-X-prime}
(\th^a)=(\th^i, \th^{N}:=\omega).
\end{equation}
In this frame, the metric tensor $g'_{ab}$ is given by $g'_{ij}=\underline{g}_{ij}, \,
g'_{NN}=-1$.

\begin{prop}
The connection forms $\omega'_a{}^b$ of $\nabla'$ are given by
\begin{equation}\label{omega-prime}
\begin{aligned}
\omega'_i{}^j&=\omega_i{}^j-\frac{1}{2}F_i{}^j\omega, 
& 
\omega'_{N}{}^j&=-\frac{1}{2}F_k{}^j\th^k, \\
\omega'_{i}{}^{N}&=\frac{1}{2}F_{ik}\th^k, 
& \omega'_{N}{}^{N}&=0.
\end{aligned}
\end{equation}
\end{prop}
\begin{proof}
It suffices to show that these 1-forms satisfy the structure equations 
$d\th^a=\th^b\wedge\omega'_b{}^a$, $dg'_{ab}=\omega'_{ab}+\omega'_{ba}$. 
By using the structure equations for $\omega_i{}^j$, we have
\begin{align*}
\th^b\wedge\omega'_b{}^i&=\th^j\wedge\Bigl(\omega_j{}^i-\frac{1}{2}F_j{}^i\omega\Bigr)+\omega\wedge\Bigl(-\frac{1}{2}F_j{}^i\th^j\Bigr) 
=\th^j\wedge\omega_j{}^i=d\th^i, \\
\th^b\wedge\omega'_b{}^{N}&=\th^j\wedge\frac{1}{2}F_{jk}\th^k=d\omega
\end{align*}
and
\[
\omega'_{ij}+\omega'_{ji}=\omega_{ij}+\omega_{ji}=dg'_{ij}, 
\quad 2\omega'_{NN}=0, \quad 
\omega'_{iN}+\omega'_{N i}=0.
\]
\end{proof}
Using \eqref{omega-prime}, the components of the divergence $\nabla'^aF_{ab}$ are computed as
\begin{equation}\label{div-F}
\nabla'^aF_{aj}=\nabla^iF_{ij}, \quad \nabla'^aF_{a N}=\frac{1}{2}|F|^2,
\end{equation}
where $|F|^2:=F_{ab}F^{ab}=F_{ij}F^{ij}$.

Next, we express the curvature tensor $R'_{abcd}$ of $g'$ in terms of the curvature $R_{ijkl}$ of $\underline{g}$ and the tensor $F$. By the symmetry of the curvature tensor, we only need to compute $R'_{ijkl}$, $R'_{N jkl}$, $R'_{N j N l}$.

\begin{prop}
The curvature tensor $R'_{abcd}$ is given by 
\begin{equation}\label{R-prime}
\begin{aligned}
R'_{ijkl}&=R_{ijkl}+\frac{1}{2}F_{ij}F_{kl}+\frac{1}{4}F_{ik}F_{jl}-\frac{1}{4}F_{jk}F_{il}, \\
R'_{Njkl}&=-\frac{1}{2}\nabla_j F_{kl}, \\
R'_{NjNl}&=\frac{1}{4}F_{jm}F_l{}^m.
\end{aligned}
\end{equation}
The Ricci and scalar curvatures of $g'$ are given by
\begin{equation}\label{ricci-prime}
R'_{ij}=R_{ij}+\frac{1}{2}F_{il}F_j{}^l,  \quad
R'_{N j}=\frac{1}{2}\nabla^i F_{ij}, \quad
R'_{NN}=\frac{1}{4}|F|^2
\end{equation}
and 
\begin{equation}\label{scalar-prime}
R'=R+\frac{1}{4}|F|^2.
\end{equation}
\end{prop}
\begin{proof}
We denote the curvature forms of $g', \underline{g}$ by 
\[
\Omega'_d{}^c=\frac{1}{2}R'_{ab}{}^c{}_d\th^a\wedge\th^b, \quad \Omega_l{}^k=\frac{1}{2}R_{ij}{}^k{}_l\th^i\wedge\th^j
\]
respectively. We compute with a coframe $(\th^i)$ on $X$ such that $\omega_i{}^j=0$ at a fixed point. 
By \eqref{omega-prime}, we have
\begin{align*}
\Omega'_l{}^k&=d\omega'_l{}^k-\omega'_l{}^p\wedge\omega'_p{}^k-
\omega'_l{}^{N}\wedge\omega'_{N}{}^k \\
&=d\Bigl(\omega_l{}^k-\frac{1}{2}F_l{}^k\omega\Bigr)
-\Bigl(\omega_l{}^m-\frac{1}{2}F_l{}^m\omega\Bigr)\wedge\Bigl(\omega_m{}^k-\frac{1}{2}F_m{}^k\omega\Bigr) \\
&\quad -\frac{1}{2}F_{li}\th^i\wedge\Bigl(-\frac{1}{2}F_j{}^k\th^j\Bigr) \\
&=\Omega_l{}^k-\frac{1}{2}\nabla_jF_l{}^k\th^j\wedge\omega-\frac{1}{2}F_l{}^kd\omega+\frac{1}{4}F_{li}F_j{}^k\th^i\wedge\th^j,
\end{align*}
from which we obtain the first and the second equations of \eqref{R-prime}. We also have
\begin{align*}
\Omega'_l{}^{N}&=d\omega'_l{}^{N}-\omega'_l{}^m\wedge\omega'_m{}^{N}-\omega'_l{}^{N}\wedge\omega'_{N}{}^{N}\\
&=\frac{1}{2}d(F_{lj}\th^j)-\Bigl(\omega_l{}^m-\frac{1}{2}F_l{}^m\omega\Bigr)\wedge\frac{1}{2}F_{mj}\th^j \\
&=\frac{1}{2}\nabla_iF_{lj}\th^i\wedge\th^j+\frac{1}{4}F_l{}^m F_{mj}\omega\wedge\th^j
\end{align*}
and the third equation of \eqref{R-prime}. The Ricci and scalar curvatures  are obtained by taking traces.
\end{proof}
\begin{prop}\label{E-prime}
If we define a symmetric tensor $E'_{ab}$ on $X'$ by
\[
E'_{ab}:=R'_{ab}+\frac{n+2}{2}g'_{ab}-\frac{1}{2}F_{ac}F_b{}^c-\frac{1}{4}
(|F|^2-2n-4)K_a K_b,
\]
then we have
\[
E'_{ij}=R_{ij}+\frac{n+2}{2}g_{ij}, \quad 
E'_{Nj}=\frac{1}{2}\nabla^i F_{ij}, \quad 
E'_{NN}=0.
\]
\end{prop}
\begin{proof}
Noting that $F_{N j}=0$, $K_N=-1$, we have by \eqref{ricci-prime}
\begin{align*}
E'_{ij}&=R'_{ij}+\frac{n+2}{2}g_{ij}-\frac{1}{2}F_{ik}F_j{}^k=R_{ij}+\frac{n+2}{2}g_{ij}, \\
E'_{N j}&=R'_{N j}=\frac{1}{2}\nabla^i F_{ij}, \\
E'_{NN}&=R'_{NN}-\frac{n+2}{2}-\frac{1}{4}(|F|^2-2n-4)=0.
\end{align*}
\end{proof}
When $n=1$ and  $F^-_{ij}=O(r^\infty)$, it holds that $\nabla^i F_{ij}=O(r^\infty)$. Thus, we obtain:
\begin{cor}\label{Einstein-Maxwell}
Assume that $n=1$ and $F^-_{ij}=O(r^\infty)$. Then, the Einstein equation 
$\mathrm{Ric}(\underline{g})+\frac{3}{2}\underline{g}=O(r^\infty)$
for $\underline{g}$ is equivalent to the equation 
\[
E'_{ab}=O(r^\infty)
\]
for $g'$. 
\end{cor}
\subsection{Translation of the self-dual equations}\label{translation-self-dual}
W now rewrite the self-dual equations $F^-_{ij}=0$ and $W^-_{ijkl}=0$ in terms of $g'$.

We fix the orientation of $X'$ so that the volume form 
\[
vol_{g'}=\frac{1}{5!}\varepsilon'_{abcde}\th^a\wedge\th^b\wedge \th^c\wedge\th^d\wedge \th^e, \quad \varepsilon'_{abcde}=\varepsilon'_{[abcde]}
\]
 of $g'$ satisfies
\[
K\lrcorner\, vol_{g'}=\pi^*vol_{g}.
\]
Then, if we set
\[
\mu_{abcd}:=K^e\varepsilon'_{eabcd},
\]
it descends to $\varepsilon_{ijkl}$: $K^a\mu_{abcd}=0,\ \mu_{ijkl}=\varepsilon_{ijkl}$ in the coframe \eqref{coframe-X-prime}. Hence, if we define a tensor $F^-_{ab}$ on $X'$ by
\[
F^-_{ab}:=\frac{1}{2}(F_{ab}-{*F}_{ab})\quad  {\rm with}\quad  {*F}_{ab}:=\frac{1}{2}\mu_{ab}{}^{cd}F_{cd},
\]
it descends to $F^-_{ij}$. Consequently, the equation $F^-_{ij}=O(r^\infty)$ is rewritten as $F^-_{ab}=O(r^\infty)$.

We set $F^2_{ab}:=F_{ac}F_b{}^c$ and define a tensor $\wh{R}_{abcd}$ on $X'$ by
\begin{align*}
\wh{R}_{abcd}&:=R'_{abcd}-K_{[a}\nabla'_{b]}F_{cd}-K_{[c}\nabla'_{d]}F_{ab} 
+\frac{1}{2}K_{c}K_{[a} F^2_{b]d}-\frac{1}{2}K_{d}K_{[a} F^2_{b]c} \\
&\quad -\frac{1}{2}F_{ab}F_{cd}-\frac{1}{4}F_{ac}F_{bd}+\frac{1}{4}F_{bc}F_{ad}.
\end{align*}

\begin{prop}
The tensor $\wh{R}_{abcd}$ descends to the curvature tensor $R_{ijkl}$ on $X$.
\end{prop}
\begin{proof}
By the first equation in \eqref{R-prime}, the components $\wh{R}_{ijkl}$ in the coframe \eqref{coframe-X-prime} agree with $R_{ijkl}$. Moreover, the tensor $\wh{R}_{abcd}$ is $S^1$-invariant and has the symmetries 
$\wh{R}_{abcd}=\wh{R}_{[ab][cd]}=\wh{R}_{cdab}$. Therefore, it suffices to show the equation $K^a \wh{R}_{abcd}=0$.

Since $K$ is a Killing vector field, we have 
\[
K^a R'_{abcd}=\nabla'_b\nabla'_c K_d=-\frac{1}{2}\nabla'_b F_{cd}, \quad 
K^b \nabla'_b F_{cd}=0.
\]
Thus, we have
\begin{align*}
K^a \wh{R}_{abcd}&=-\frac{1}{2}\nabla'_b F_{cd}+\frac{1}{2}\nabla'_b F_{cd}
+K_{[c}(\nabla'_{d]}K^a)F_{ab}-\frac{1}{4}K_c F^2_{bd}+\frac{1}{4}K_d F^2_{bc} \\
&=0.
\end{align*}
\end{proof}

When $\underline{g}$ satisfies $R_{ij}+\frac{3}{2}\underline{g}_{ij}=O(r^\infty)$, the Weyl curvature of $\underline{g}$ is given by 
\[
W_{ijkl}=R_{ijkl}+\frac{1}{2}(\underline{g}_{ik}\underline{g}_{jl}-\underline{g}_{jk}\underline{g}_{il})+O(r^\infty).
\]
On $X'$, we have
\begin{align*}
\underline{g}_{ik}\underline{g}_{jl}-\underline{g}_{jk}\underline{g}_{il}&=(g'_{ac}+K_aK_c)(g'_{bd}+K_bK_d)-(g'_{bc}+K_bK_c)(g'_{ad}+K_aK_d) \\
&=2 g'_{c[a}g'_{b]d}+2g'_{c[a}K_{b]}K_d-2g'_{d[a}K_{b]}K_c.
\end{align*}
Here, the equation is understood to mean that the right-hand side descends to the left-hand side.
Hence, if we set 
\begin{align*}
 \mathring L_{abcd}&:=\wh{R}_{abcd}+g'_{c[a}g'_{b]d}+g'_{c[a}K_{b]}K_d-g'_{d[a}K_{b]}K_c \\
&=R'_{abcd}-K_{[a}\nabla'_{b]}F_{cd}-K_{[c}\nabla'_{d]}F_{ab} 
+\frac{1}{2}K_{c}K_{[a} F^2_{b]d}-\frac{1}{2}K_{d}K_{[a} F^2_{b]c}\\
&\quad -\frac{1}{2}F_{ab}F_{cd}+\frac{1}{2}F_{c[	a}F_{b]d}+g'_{c[a}g'_{b]d}+g'_{c[a}K_{b]}K_d-g'_{d[a}K_{b]}K_c,
\end{align*}
it descends to the tensor $W_{ijkl}+O(r^\infty)$ and the equation $W^-_{ijkl}=O(r^\infty)$ is equivalent to $\mathring L^-_{abcd}=O(r^\infty)$, where 
\[ 
\mathring L^-_{abcd}:=\frac{1}{2}( \mathring L_{abcd}-*\mathring L_{abcd})\quad  {\rm with}\quad 
*\mathring L_{abcd}:=\frac{1}{2}\mu_{ab}{}^{pq} \mathring L_{pqcd}.
\]
To simplify the formula slightly, we introduce the tensor
\begin{align*}
L_{abcd}&:= \mathring L_{abcd}+\frac{1}{2}F_{ab}F_{cd} \\
&=R'_{abcd}-K_{[a}\nabla'_{b]}F_{cd}-K_{[c}\nabla'_{d]}F_{ab} 
+\frac{1}{2}K_{c}K_{[a} F^2_{b]d}-\frac{1}{2}K_{d}K_{[a} F^2_{b]c}\\
&\quad +\frac{1}{2}F_{c[	a}F_{b]d}+g'_{c[a}g'_{b]d}+g'_{c[a}K_{b]}K_d-g'_{d[a}K_{b]}K_c
\end{align*}
and define $L^-_{abcd}$ in the same way as $\mathring L^-_{abcd}$. Then, under the condition $F^-_{ab}=O(r^\infty)$, we have
\[
 L^-_{abcd}= \mathring L^-_{abcd}+O(r^\infty)
\]
and hence the self-dual equation $W^-_{ijkl}=O(r^\infty)$ is equivalent to 
\[
L^-_{abcd}=O(r^\infty).
\]
Combining this with Corollary \ref{Einstein-Maxwell} , we obtain Theorem \ref{Poincare}.

\section{Ambient metric for the Fefferman metric}

\subsection{$S^1$-invariant straight pre-ambient metrics}
We will construct the ambient metric corresponding to the $S^1$-invariant Poincar\'e metric that we constructed in the previous section.  We begin by recalling the notion of a straight pre-ambient metric for conformal manifolds introduced by Fefferman-Graham \cite{FG}, and adapt it to our setting.

Let $(\mathcal{C}, [g^{\mathrm{F}}_\th])$ be the Fefferman space over a $(2n+1)$-dimensional strictly pseudoconvex CR manifold $M$. The {\it metric bundle} $\mathcal{G}\to \mathcal{C}$ is the $\mathbb{R}_+$-bundle whose sections correspond to representative metrics in $[g^{\mathrm{F}}_\th]$. 
The bundle is equipped with the tautological $2$-tensor $\boldsymbol{g}$, which is expressed as $\boldsymbol{g}=t^2g^{\mathrm{F}}_\th$ with respect to the trivialization $\mathcal{G}\cong\mathcal{C}\times (0, \infty)_t$ by a representative metric $g^{\mathrm{F}}_\th$. 
Let $\wt{\mathcal{G}}$ be a $(2n+4)$-dimensional manifold which contains $\mathcal{G}$ as a hypersurface and is equipped with an $\mathbb{R}_+$-action, denoted by $\d_s\ (s\in\mathbb{R}_+)$, which restricts to the canonical dilation on $\mathcal{G}$. We assume that there exists an $\mathbb{R}_+$-equivaraint diffeomorphism
\[
\wt{\mathcal{G}}\cong \mathcal{G}\times(-\e_0, \e_0)
\]
which restricts to the identification $\mathcal{G}\cong\mathcal{G}\times\{0\}$. Here, $\mathbb{R}_+$ acts only on the first component of $\mathcal{G}\times(-\e_0, \e_0)$. We call $\wt{\mathcal{G}}$ the {\it ambient space} of $(\mathcal{C}, [g^{\mathrm{F}}_\th])$.
\begin{dfn}
A pseudo-Riemannian metric $\wt g$ of signature $(2n+2, 2)$ on $\wt{\mathcal{G}}$ is called a {\it pre-ambient  metric} if it satisfies the following conditions:
\begin{itemize}
\item[(i)]  $\wt g$ is homogeneous of degree $2$: $\d_s^*\wt g=s^2 \wt g,\ s\in\mathbb{R}_+$.
\item[(ii)] The restriction of $\wt g$ to $T\mathcal{G}$ agrees with the tautological $2$-tensor 
$\boldsymbol{g}$.
\end{itemize}
\end{dfn}
Let $\wt g$ be a pre-ambient metric and $\wt\nabla$ the Levi-Civita connection of $\wt g$. Let 
\[
E:=\frac{d}{ds}\Big|_{s=1}\d_s
\]
denote the Euler vector field (the infinitesimal generator of $\mathbb{R}_+$-action) on $\wt{\mathcal{G}}$. The pre-ambient metric $\wt g$ is said to be {\it straight} if it satisfies
\[
\wt\nabla E=\mathrm{id};
\]
see {\cite[Proposition 2.4]{FG}} for other characterizations of the straight condition.

Now we will incorporate the $S^1$-bundle structure into the picture. Since the conformal structure $[g^{\mathrm{F}}_\th]$ is $S^1$-invariant, the $S^1$-action lifts naturally to the metric bunle $\mathcal{G}$ and it preserves the tensor $\boldsymbol{g}$. Moreover, the $S^1$-action commutes with the $\mathbb{R}_+$-action, so we have a $\mathbb{C}^*$-action on $\mathcal{G}$. We assume that the ambient space $\wt{\mathcal{G}}$ also admits a $\mathbb{C}^*$-action, also denoted by $\d$, and a $\mathbb{C}^*$-equivariant diffeomorphism $\wt{\mathcal{G}}\cong \mathcal{G}\times(-\e_0, \e_0)$. We consider a  straight pre-ambient metric $\wt g$ which is $S^1$-invariant:
\[
\d_\lambda^* \wt g=|\lambda|^2 \wt g, \quad \lambda\in\mathbb{C}^*.
\]
Let
\[
\wt K:=\frac{n+2}{2}\frac{d}{d s}\Big|_{s=0}\d_{e^{is}}
\]
which is a multiple of the infinitesimal generator of the $S^1$-action on $\wt{\mathcal{G}}$.  Since $\wt K$ is a Killing vector field, the 2-tensor
\[
\wt F_{AB}:=-2\wt\nabla_A \wt K_B
\]
is skew-symmetric.
\subsection{Relation to $S^1$-invariant Poincar\'e metrics}\label{ambient-Poincare}
The Poincar\'e metric associated with an $S^1$-invariant straight pre-ambient metric $\wt g$ is
defined as follows: A choice of representative metric $g^{\mathrm{F}}_\th$ determines a unique $\mathbb{R}_+$-equivariant diffeomorphism $\wt{\mathcal{G}}\cong \mathcal{G}\times(-\e_0, \e_0)_\rho\cong \mathcal{C}\times(0, \infty)_t\times (-\e_0, \e_0)_\rho$ which is identity on $\mathcal{G}\cong \mathcal{G}\times\{0\}$ such that $\wt g$ is written in the normal form 
\[
\wt g=2\rho dt^2+2tdtd\rho+t^2 g_\rho,
\]
where $g_\rho$ is a smooth family of Lorentzian metrics on $\mathcal{C}$ with $g_0=g^{\mathrm{F}}_\th$. 
Since $g^{\mathrm{F}}_\th$ and $\wt g$ are $S^1$-invariant, the uniqueness implies that the diffeomorphism is $S^1$-equivariant. In particular, $\wt K$ agrees with the trivial extension of $K$ to $\mathcal{C}\times(0, \infty)_t\times (-\e_0, \e_0)_\rho$, and is perpendicular to the Euler field $E=t\partial_t$: 
\begin{equation}\label{E-K}
\wt g(E, \wt K)=0.
\end{equation}
 The squared norm 
\[
\wt \rho:=\wt g(E, E) \bigl(=2t^2\rho\bigr)
\]
of the Euler field gives an $S^1$-invariant defining function of $\mathcal{G}\subset \wt{\mathcal{G}}$ homogeneous of degree $2$. On the hypersurface $\mathcal{H}:=\{\wt\rho=-1\}\subset \wt{\mathcal{G}}$, 
the projection 
\[
\tilde\pi\colon\mathcal{G}\times (-\e_0, \e_0)_\rho \longrightarrow \mathcal{C}\times (-\e_0, \e_0)_\rho
\]
restricts to a diffeomorpshim $\tilde\pi|_{\mathcal{H}}\colon 
\mathcal{H}\to \mathcal{C}\times (-\e_0, 0)_\rho$.
We define a Lorentzian metric $g'$ on $\mathcal{C}\times (0, \e)_r\ (\e=\sqrt{2\e_0})$ by
\[
g':=4\bigl((\chi\circ \tilde\pi|_{\mathcal{H}})^{-1}\bigr)^* \wt g,
\]
where $\chi\colon \mathcal{C}\times (-\e_0, 0)_\rho\to 
\mathcal{C}\times (0, \e)_r$ is the differemorphism given by
\[
\chi (x, \rho):=(x, \sqrt{-2\rho}).
\]
Then, since the map $(\chi\circ \tilde\pi|_{\mathcal{H}})^{-1}$ is given by $t=r^{-1}$, $\rho=-\frac{1}{2}r^2$,   the metric $g'$ is written as
\[
g'=4\frac{dr^2}{r^2}+4r^{-2}g_{-\frac{1}{2}r^2}.
\]
Thus, $g'$ becomes an $S^1$-invariant even Poincar\'e metric with the conformal infinity $(\mathcal{C}, [g^{\mathrm{F}}_\th])$ if one has $g'(K, K)=-1$.
Introducing a new variable $u$ by 
\[
u:=rt=\sqrt{-2\rho}\,t,
\]
we  have the diffeomorphism
\begin{equation}\label{product}
\{\wt\rho<0\}\cong \mathcal{H}\times (0, \infty)_u\cong \mathcal{C}\times (0, \e)_r\times (0, \infty)_u.
\end{equation}
Under this identification, $\wt g|_{\{\wt\rho<0\}}$ is expressed as a warped product metric:
\begin{equation}\label{warped}
\wt g=\frac{u^2}{4}g'-du^2=-\frac{\wt\rho}{4}g'-du^2.
\end{equation}
It follows that
\begin{equation}\label{K-norm}
\wt g(\wt K, \wt K)=-\frac{\wt\rho}{4}g'(K, K).
\end{equation}
Thus, the equation $g'(K, K)=-1$ is equivalent to the equation
\[
 \wt g(\wt K, \wt K)=\frac{\wt \rho}{4}
\]
on $\{\wt\rho<0\}$. We assume that $\wt g$ satisfies this equation.

Let $g^{\mathrm{F}}_{\wh\th}$ be another representative metric. Then, we have the associated 
diffeomorphisms $\wt{\mathcal{G}}\cong \mathcal{G}\times (-\hat{\e}_0, \hat{\e}_0)_{\wh\rho}$, $\mathcal{H}\cong \mathcal{C}\times (0, \hat\e)_{\wh r}$ and an $S^1$-invariant Poincar\'e
 metric on $\mathcal{C}\times (0, \hat\e)_{\wh r}$ as above. The diffeomorphism 
 $\mathcal{C}\times (0, \e)_r\cong\mathcal{C}\times (0, \hat\e)_{\wh r}$ through $\mathcal{H}$ gives an $S^1$-equivariant isometry which extends to a diffeomorphism $\mathcal{C}\times [0, \e)_r\cong \mathcal{C}\times [0, \hat\e)_{\wh r}$ fixing points on $\mathcal{C}\times\{0\}$.
Thus, identifying them by this canonical $S^1$-equivariant isometry, we obtain a manifold $X'\cong \mathcal{C}\times (0, \e)$ endowed with an $S^1$-action and an  $S^1$-invariant even Poincar\'e metric $g'$.

When $n=1$, the condition that $g'$ agrees with the one given by Theorem \ref{Poincare} completely determines the expansion of $g_\rho$ in the normal form of $\wt g$.  
We will rewrite the equations for $g'$ in Theorem \ref{Poincare} in terms of curvature quantities of $\wt g$ and prove Theorem \ref{ambient}.
\subsection{Ambient expression of $E'_{ab}$}
We first deal with the tensor $E'_{ab}$. We fix a representative metric $g^{\mathrm{F}}_\th$ and
work on 
\begin{equation*}
\{\wt\rho<0\}\cong X'\times (0, \infty)_u\cong \mathcal{C}\times (0, \e)_r\times (0, \infty)_u
\end{equation*}
as in \eqref{product}. Note that since we consider the formal expansion of $\wt g$, identities on $\{\wt\rho<0\}$ can be regarded as identities on whole $\wt{\mathcal{G}}$.
\begin{prop}
We have
\begin{equation}\label{E-K-F}
E^A \wt F_{AB}=-2\wt K_B, \quad \wt K^A \wt F_{AB}=\frac{1}{2}E_B.
\end{equation}
\end{prop}
\begin{proof}
By \eqref{E-K} and the straight condition $\wt\nabla _AE^B=\d_A{}^B$, we have 
\[
E^A \wt F_{AB}=
2E^A\wt \nabla_B \wt K_A=-2(\wt \nabla_BE^A) \wt K_A=-2\wt K_B.
\]
By using \eqref{warped} and the formula of the covariant differentiation for warped product metrics, we have
\begin{equation}\label{derivative-warped}
\wt \nabla_V W=\nabla'_V W+\frac{1}{4}g'(V, W) E
\end{equation}
for any vector fields $V, W\in \Gamma(TX')\subset \Gamma(T\wt{\mathcal{G}})$.
Thus, we have $\wt\nabla_{\wt K}\wt K=-\frac{1}{4}E$, which is equivalent to $\wt K^A \wt F_{AB}=\frac{1}{2}E_B$.
\end{proof}
In the sequel, we adopt the following notation: For a tensor 
$U_{A_1\cdots A_m}$ on $\wt{\mathcal{G}}$, and a tensor $V_{a_1\cdots a_m}$ on $X'$, we write as
\[
U_{A_1\cdots A_m}=V_{a_1\cdots a_m}
\]
if $U_{A_1\cdots A_m}$ descends to $V_{a_1\cdots a_m}$; namely, under the identification $\{\wt\rho<0\}\cong X'\times (0, \infty)_u$, the  tensor $U_{A_1\cdots A_m}$ has vanishing components except for 
 $U_{a_1\cdots a_m}$ which agrees with $V_{a_1\cdots a_m}$.
For example, we have
\begin{equation}\label{g-prime-ambient}
 g'_{ab}=-\frac{4}{\wt\rho}\,\wt g_{AB}+\frac{4}{\wt\rho^{\,2}}E_A E_B
\end{equation}
by \eqref{warped}.

We set 
\[
\Omega_{AB}:=2E_{[A}\wt K_{B]}
\]
and define a skew symmetric 2-tensor $\wt F'_{AB}$ by
\begin{equation}\label{F-prime}
\wt F'_{AB}:=-\frac{4}{\wt\rho}\,\wt F_{AB}-\frac{8}{\wt\rho^{\,2}}\Omega_{AB}.
\end{equation}

\begin{prop}
We have 
\begin{equation}\label{F-ambient}
\wt F'_{AB}=F_{ab}.
\end{equation}
\end{prop}
\begin{proof}
By \eqref{E-K-F}, we have $E^A\wt F'_{AB}=0$. Since $\wt F'_{AB}$ is homogeneous of degree 0, it descends to a 2-form on $X'$. 
We take arbitrary vector fields $V, W$ on $X'$ and extend them trivially to $\{\wt\rho<0\}\cong X'\times (0, \infty)_u$. Then, by \eqref{derivative-warped}, we have 
\begin{align*}
\wt F'_{AB}V^AW^B&=-\frac{4}{\wt\rho}\,\wt F_{AB}V^A W^B 
=\frac{8}{\wt\rho}\,\wt g(\wt\nabla _V \wt K, W) 
=\frac{8}{\wt\rho}\, \wt g(\nabla'_V K, W) \\
&=-2g'(\nabla'_V K, W)=F_{ab}V^a W^b.
\end{align*}
Thus, $\wt F'_{AB}$ descends to $F_{ab}$.
\end{proof}
By this proposition, we have
\begin{equation}\label{F2-ambient}
F_{ac}F_b{}^c=-\frac{\wt\rho}{4}\,\wt F'_{AC}\wt F'_B{}^C
=-\frac{4}{\wt\rho}\,\wt F_{AC}\wt F_B{}^C+\frac{16}{\wt\rho^{\,2}}\wt K_A\wt K_B+\frac{4}{\wt\rho^{\,2}}\,E_A E_B.
\end{equation}
Taking the trace yields
\begin{equation}\label{norm-F}
|F|^2=\frac{\wt\rho^{\,2}}{16}|\wt F'|^2=|\wt F|^2-2.
\end{equation}
By \eqref{g-prime-ambient}, we also have
\begin{equation}\label{F-g}
F_{ac}F_b{}^c-g'_{ab}=
-\frac{4}{\wt\rho}(\wt F_{AC}\wt F_B{}^C-\wt g_{AB})+
\frac{16}{\wt\rho^{\,2}}\wt K_A \wt K_B.
\end{equation}

Now we set 
\[
\wt E_{AB}:=\wt R_{AB}+\frac{2}{\wt\rho}(\wt F_{AC}\wt F_B{}^C-\wt g_{AB})
-\frac{4}{\wt\rho^{\,2}}(|\wt F|^2-2n-4)\wt K_A \wt K_B,
\]
where $\wt R_{AB}=\wt R_{CA}{}^C{}_{B}$ is the Ricci tensor of $\wt g$.
\begin{prop}
We have $\wt E_{AB}=E'_{ab}$.
\end{prop}
\begin{proof}
The straight condition $\wt\nabla _AE^B=\d_A{}^B$ implies $E^A\wt R_{ABCD}=0$, so the curvature tensor of $\wt g$ has only the horizontal components $\wt R_{abcd}$. The formula for the curvature of warped product metrics gives
\begin{equation}\label{wt-R}
\wt R_{abcd}=-\frac{\wt\rho}{4}R'_{abcd}-\frac{\wt\rho}{16}(g'_{ac}g'_{bd}-g'_{bc}g'_{ad}).
\end{equation}
Thus, we have
\begin{equation}\label{ambient-ricci}
\wt R_{AB}=R'_{ab}+\frac{n+1}{2}g'_{ab}.
\end{equation}
From \eqref{norm-F}, \eqref{F-g}, \eqref{ambient-ricci}, we have
\begin{align*}
E'_{ab}&=R'_{ab}+\frac{n+2}{2}g'_{ab}-\frac{1}{2}F_{ac}F_b{}^c-\frac{1}{4}
(|F|^2-2n-4)K_a K_b \\
&=\wt R_{AB}+\frac{2}{\wt\rho}(\wt F_{AC}\wt F_B{}^C-\wt g_{AB})
-\frac{4}{\wt\rho^{\,2}}(|\wt F|^2-2n-4)\wt K_A \wt K_B.
\end{align*}
\end{proof}
By this proposition, the equation $E'_{ab}=O(r^\infty)$ is equivalent to $\wt E_{AB}=O(\rho^\infty)$.


\subsection{Ambient expressions of $L^-_{abcd}, F^-_{ab}$}
Next, we rewrite the equations $L^-_{abcd}=0$ and $F^-_{ab}=0$.
\begin{prop}
On $\{\wt \rho<0\}\subset \wt{\mathcal{G}}$, we have
\begin{equation}\label{ambient-nabla-F}
\nabla'_a F_{bc}=-\frac{4}{\wt\rho}\wt\nabla_A \wt F_{BC}-\frac{16}{\wt\rho^{\,2}}\wt g_{A[B}\wt K_{C]}-\frac{16}{\wt\rho^{\,3}}E_A\wt K_{[B}E_{C]}.
\end{equation}
\end{prop}
\begin{proof}
We denote the right-hand side of \eqref{ambient-nabla-F} by $H_{ABC}$. By \eqref{E-K-F}, we have
\[
E^B\wt \nabla_A \wt F_{BC}=\wt\nabla_A(E^B \wt F_{BC})-(\wt \nabla_AE^B)\wt F_{BC} =-2\wt\nabla_A\wt K_C-\wt F_{AC}=0,
\]
and hence
\begin{align*}
E^B H_{ABC}=-\frac{8}{\wt\rho^{\,2}}E_A\wt K_C+\frac{8}{\wt\rho^{\,3}}\cdot \wt\rho\, E_A\wt K_C 
=0.
\end{align*}
Since $H_{ABC}$ is homogeneous of degree 0 and satisfies $H_{ABC}=H_{A[BC]}, H_{[ABC]}=0$, this implies that $H_{ABC}$ descends to the tensor $H_{abc}$ on $X'$. Since $\wt K$ is a Killing vector field, it satisfies
\[
-2\wt K^D \wt R_{DABC}=-2\wt\nabla_A \wt\nabla_B \wt K_C=\wt\nabla_A\wt F_{BC}.
\]
Then, we have
\begin{align*}
\wt\nabla_a\wt F_{bc}=-2\wt K^D \wt R_{Dabc}&=\frac{\wt\rho}{2}K^d R'_{dabc}+\frac{\wt\rho}{8}
(K_b g'_{ac}-K_c g'_{ab}) \\
&=-\frac{\wt\rho}{4}\nabla'_a F_{bc}+\frac{\wt\rho}{8}
(K_b g'_{ac}-K_c g'_{ab})
\end{align*}
by \eqref{wt-R}. Moreover, noting that 
$\wt g_{ab}=(-\wt\rho/4)g'_{ab}$ and $\wt K_c=(-\wt\rho/4)K_c$, we see that 
\[
\frac{16}{\wt\rho^{\,2}}\wt g_{a[b}\wt K_{c]}=g'_{a[b}K_{c]}.
\]
Thus we obtain $H_{abc}=\nabla'_a F_{bc}$.
\end{proof}
We set $\wt F^2_{AB}:=\wt F_{AC}\wt F_{B}{}^C$. Then, by using \eqref{g-prime-ambient}, \eqref{F-ambient}, \eqref{wt-R} and \eqref{ambient-nabla-F}, we have
\[
L_{abcd}=\wt L_{ABCD},
\]
where
\begin{equation}\label{L-tilde}
\begin{aligned}
\wt L_{ABCD}&:=-\frac{4}{\wt\rho}\wt R_{ABCD}+\frac{8}{\wt\rho^{\,2}}(\wt g_{C[A}\wt g_{B]D}+\wt F_{C[A}\wt F_{B]D} ) \\
&\quad +\frac{16}{\wt\rho^{\,3}}(E_{[A}\wt g_{B][C}E_{D]}+4 \wt K_{[A}\wt F^2_{B][C}\wt K_{D]})\\
&\quad -\frac{16}{\wt\rho^{\,2}}(\wt K_{[A}\wt\nabla_{B]} \wt F_{CD}+\wt K_{[C}\wt\nabla_{D]} \wt F_{AB})
 \\
&\quad  +\frac{16}{\wt\rho^{\,3}}
(\wt F_{C[A}\Omega_{B]D}-\wt F_{D[A}\Omega_{B]C}).
\end{aligned}
\end{equation}

We choose an orientation of $\wt{\mathcal{G}}$ so that $du\wedge vol_{g'}>0$ on $\{\wt\rho<0\}\cong X'\times (0, \infty)_u$. Then, the volume form of the ambient metric 
\[
vol_{\wt g}=\frac{1}{6!}\wt{\varepsilon}_{ABCDEF}\,\th^A\wedge\th^B\wedge\th^C\wedge\th^D\wedge\th^E\wedge\th^F, \quad 
\wt\varepsilon_{ABCDEF}=\wt\varepsilon_{[ABCDEF]}
\]
is given by
\[
vol_{\wt g}=du\wedge \Big(-\frac{\wt \rho}{4}\Bigr)^{5/2}vol_{g'}
=\frac{1}{32}(-\wt\rho)^{5/2}du\wedge vol_{g'}
\]
on $\{\wt\rho<0\}\cong X'\times (0, \infty)_u$. Thus, we have
\[
E\lrcorner\, vol_{\wt g}=\frac{1}{32}(-\wt\rho)^{5/2}u\cdot vol_{g'}=\frac{1}{32}(-\wt\rho)^{3}vol_{g'}
\]
or equivalently,
\[
E^A \wt\varepsilon_{ABCDEF}=\frac{1}{32}(-\wt\rho)^{3} \varepsilon'_{bcdef}.
\]
If we introduce the ambient $4$-form
\[
\wt \mu_{CDEF}:=E^A\wt K^B \wt\varepsilon_{ABCDEF},
\]
we have
\[
\wt \mu_{CDEF}=\frac{1}{32}(-\wt\rho)^{3} K^b\varepsilon'_{bcdef}
=\frac{1}{32}(-\wt\rho)^{3} \mu_{cdef}
\]
by using the preceding identity.
Raising the indices gives
\[
\wt \mu_{CD}{}^{EF}=\left(\frac{4}{\wt\rho}\right)^2\frac{1}{32}(-\wt\rho)^{3} \mu_{cd}{}^{ef}=\frac{-\wt\rho}{2}\mu_{cd}{}^{ef}.
\]
Hence we have
\[
L^-_{abcd}=\wt L^-_{ABCD}
\]
with
\begin{equation}\label{tilde-L-minus}
\wt L^-_{ABCD}:=\frac{1}{2}(\wt L_{ABCD}-*\wt L_{ABCD}),\quad *\wt L_{ABCD}:=\frac{-1}{\wt\rho}\,\wt\mu_{AB}{}^{PQ}\wt L_{PQCD}.
\end{equation}
Thus, the equation $L^-_{abcd}=O(r^\infty)$ is equivalent to $\wt L^-_{ABCD}=O(\rho^\infty)$.
Similarly, we have
\[
F^-_{ab}=(\wt F')^-_{AB}
\]
with
\begin{equation}\label{tilde-F-minus}
(\wt F')^-_{AB}
:=\frac{1}{2}({\wt F'}_{AB}-*\wt F'_{AB}), \quad *{\wt F'}_{AB}:=\frac{-1}{\wt\rho}\,{\wt\mu}_{AB}{}^{PQ}\wt F'_{PQ},
\end{equation}
so the equation $F^-_{ab}=O(r^\infty)$ is equivalent to $(\wt F')^-_{AB}=O(\rho^\infty)$.
\bigskip

{\it Proof of Theorem \ref{ambient}}\quad Let $g^{\mathrm{F}}_\th\in[g^{\mathrm{F}}_\th]$ be a representative metric and $g'$ the associated $S^1$-invariant even  Poincar\'e metric on $X'=\mathcal{C}\times (0, \e)_r$ given by Theorem 
\ref{Poincare}. Since $g'$ is even, it can be written in the normal form
\[ 
g'=4\frac{dr^2}{r^2}+4r^{-2}g_{-\frac{1}{2}r^2}
\]
with a smooth family $g_\rho$ of $S^1$-invariant  Lorentzian metrics on $\mathcal{C}$ with $g_0=g^{\mathrm{F}}_\th$. If we define an $S^1$-invariant straight pre-ambient metric $\wt g$ on 
$\wt{\mathcal{G}}=\mathcal{G}\times (-\e_0, \e_0)_\rho\cong \mathcal{C}\times (0, \infty)_t\times  (-\e_0, \e_0)_\rho$ by 
\[
\wt g=2\rho dt^2+2tdtd\rho+t^2 g_\rho,
\]
it satisfies the equations in Theorem \ref{ambient} by the previous computations, and conversely these equations uniquely determine the expansion of $g_\rho$. 

Suppose that $\wt{g}{\,'}$ is another $S^1$-invariant straight pre-ambient metric on 
$\mathcal{G}\times (-\e'_0, \e'_0)_{\rho'}$ satisfying the same equations. By the general theory of ambient metric, there exists a unique $\mathbb{R}_+$-equivariant orientation preserving diffeomorphism $\Phi\colon\mathcal{G}\times (-\e_0, \e_0)_\rho\to \mathcal{G}\times (-\e'_0, \e'_0)_{\rho'}$, fixing points on $\mathcal{G}\times\{0\}$, such that $\Phi^*\wt{g}{\,'}$ is in the normal form with respect to $g^{\mathrm{F}}_\th$. The $S^1$-invariance of $g^{\mathrm{F}}_\th$, $\wt{g}{\,'}$ implies that $\Phi$ is $S^1$-equivariant and hence $\mathbb{C}^*$-equivariant.  Since $\Phi^*\wt{g}{\,'}$ also satisfies the equations in Theorem \ref{ambient}, it must have the same expansion as $\wt g$. This proves Theorem \ref{ambient}. \qed

\subsection{Construction of CR GJMS operators in dimension three}\label{GJMS-construction}
Since our ambient metric $\wt g$ is determined to infinite order, it enables us to construct the CR GJMS operators
\[
P_{w, w'}\colon \mathcal{E}(w, w')\longrightarrow \mathcal{E}(w-k, w'-k)\quad (k=w+w'+2\in\mathbb{N}_+)
\]
on CR 3-manifolds without an upper bound of $k$. 

We locally define a $3$-fold covering of $\mathcal{C}$ by 
\[
\mathcal{C}^{1/3}:=\bigl(\mathcal{E}(-1, 0)\setminus\{0\}\bigr)/\mathbb{R}_+,
\]
and consider the pull-back of the Fefferman metric via the covering map $\mathcal{C}^{1/3}\to\mathcal{C}$, for which we keep the same notation $[g^{\mathrm{F}}_\th]$. Then, the metric bundle $\mathcal{G}^{1/3}$ for $(\mathcal{C}^{1/3}, [g^{\mathrm{F}}_\th])$ is a $3$-fold covering of $\mathcal{G}$. 
An $S^1$-invariant section $t$ of $\mathcal{E}(-1, 0)\setminus\{0\}\to \mathcal{C}^{1/3}$ 
defines a CR scale $|t|^{-2}\in\mathcal{E}(1, 1)$ and hence a representative metric $g^{\mathrm{F}}_\th$. Since a  rescaling $\wh t=e^\Upsilon t\ (\Upsilon\in C^\infty(M))$ results in the rescaling 
$g^{\mathrm{F}}_{\wh\th}=e^{2\Upsilon}g^{\mathrm{F}}_\th$,  we can identify $\mathcal{G}^{1/3}$ with $\mathcal{E}(-1, 0)\setminus\{0\}$. We  also have the $S^1$-invariant ambient metric $\wt g$ for $(\mathcal{C}^{1/3}, [g^{\mathrm{F}}_\th])$ on the $3$-fold covering ${\wt{\mathcal{G}}}^{1/3}\cong {\mathcal{G}}^{1/3}\times (-\e_0, \e_0)_\rho$ of $\wt{\mathcal{G}}$. 

Now we apply the GJMS construction \cite{GJMS, GG} to this setting. Let $\wt\Delta=-\wt\nabla_A\wt\nabla^A$ be the Laplacian of $\wt g$. We identify a CR density $f\in\mathcal{E}(w, w')\ (k=w+w'+2\in\mathbb{N}_+)$ with a homogeneous function on $\mathcal{G}^{1/3}$ and extend it arbitrarily to a function $\wt f$ on ${\wt{\mathcal{G}}}^{1/3}$ having the same homogeneity: $\d_\lambda^*\wt f=\lambda^w\bar{\lambda}^{w'} \wt f$. Then, since $\wt g$ satisfies $\d_\lambda^* \wt g=|\lambda|^2 \wt g$, we have
\[
\d_\lambda^*\bigl(\wt\Delta^k \wt f\,\bigr)=|\lambda|^{-2k}\lambda^w\bar{\lambda}^{w'} \wt\Delta^k \wt f\quad (\lambda\in\mathbb{C}^*),
\]
and the argument of GJMS \cite{GJMS} shows that the CR density
\[
P_{w, w'}f:=\bigl(\wt\Delta^k \wt f\,\bigr)|_{\mathcal{G}^{1/3}}\in\mathcal{E}(w-k, w'-k)
\]
is independent of the choice of the homogeneous extension $\wt f$ and has principal part $\Delta_b^k$; see also the proof of {\cite[Theorem 1.1]{GG}}. We emphasize that the proof of the extension-independence in \cite{GJMS} uses only the straightness of the ambient metric, and hence applies to the ambient metric constructed here. Since the coefficients of the expansion of $\wt g$ have universal expressions in terms of the Tanaka-Webster curvature and torsion, so does $P_{w, w'}$. Thus, we obtain Theorem \ref{GJMS-three-dim}.

\medskip

\noindent {\bf Acknowledgment} 
This  work was partially supported by JSPS KAKENHI Grant Number 22K13922.
The author used ChatGPT to assist in checking the arguments and computations.

\end{document}